\documentclass[a4paper]{article}

\usepackage[latin1]{inputenc}  
\usepackage{graphicx}
\usepackage[pdftex,bookmarks]{hyperref}
\usepackage{amsmath}    
\usepackage{amssymb}
\usepackage{amsmath,amsthm}
\usepackage{textcomp}
\usepackage[all]{xy}

\newtheorem{teo}{Theorem}[section]
\newtheorem{defi}{Definition}[section]

\newtheorem{cor}{Corollary}[section]
\newtheorem{prop}{Proposition}[section]

\newtheorem{rem} {Remark}[section]

\DeclareMathOperator{\im}{im}
\DeclareMathOperator{\ind}{ind}

\DeclareMathOperator{\ck}{coker}
\DeclareMathOperator{\reg}{reg}
\DeclareMathOperator{\vol}{vol}
\DeclareMathOperator{\sing}{sing}

\title{\huge On the $L^2$-Poincar\'e duality  for  incomplete riemannian manifolds: a general construction with applications}
\author{Francesco Bei  \bigskip \\
Institute f\"ur Mathematik, Humboldt Universit\"at zu Berlin,\\ E-mail addresses: \ bei@math.hu-berlin.de\     francescobei27@gmail.com }
\date{}

\begin{document}

\maketitle
\maketitle

\begin{abstract}
Let $(M,g)$ be an  open, oriented and incomplete riemannian manifold of dimension $m$. Under some general conditions we show the existence of  a Hilbert complex $(L^2\Omega^i(M,g),d_{\mathfrak{M},i})$ such that its cohomology groups, labeled with $H^i_{2,\mathfrak{M}}(M,g)$, satisfy the following properties:
\begin{itemize}
\item $H^i_{2,\mathfrak{M}}(M,g)=\ker(d_{max,i})/\im(d_{min,i})$
\item $H^i_{2,\mathfrak{M}}(M,g)\cong H^{m-i}_{2,\mathfrak{M}}(M,g)$ (Poincar\'e duality holds)
\item There exists a  well defined and non degenerate pairing: $$H^i_{2,\mathfrak{M}}(M,g)\times H^{m-i}_{2,\mathfrak{M}}(M,g)\longrightarrow \mathbb{R},\ ([\omega],[\eta])\longmapsto \int_{M}\omega\wedge \eta$$
\item If $(L^2\Omega^i(M,g),d_{\mathfrak{M},i})$ is  a Fredholm complex then every  closed extension of the de Rham complex $(\Omega^i_c(M),d_i)$ is a Fredholm complex and,  for each $i=0,...,m$, the quotient $\mathcal{D}(d_{max,i})/\mathcal{D}(d_{min,i})$ is a finite dimensional vector space. 
\end{itemize}
\end{abstract}
\vspace{1 cm}

\noindent\textbf{Keywords}: $L^2$-cohomology, Poincar\'e duality, Incomplete riemannian manifolds, Fredholm complexes. \vspace{1 cm}

\noindent\textbf{Mathematics subject classification}:  58J10, 14F40, 14F43.

\section*{Introduction}
Poincar\'e duality is one of the best known and most important properties of the de Rham cohomology on a closed and  oriented smooth manifold $M$.  Using the pairing induced by the wedge product we have:
\begin{equation}
\label{topobirbone}
H^i_{dR}(M)\times H^{m-i}_{dR}(M)\longrightarrow \mathbb{R},\ ([\omega],[\eta])\longmapsto \int_{M}\omega\wedge \eta.
\end{equation}
Poincar\'e duality says that \eqref{topobirbone} induces an isomorphism between $H^i_{dR}(M)$ and $(H^{m-i}_{dR}(M))^*$ for all $i=0,...,m$, where $m$ is the dimension of $M$ . Besides the previous isomorphism, putting a riemannian metric $g$ on $M$ and using the results coming from Hodge theory, we have also the following  isomorphisms: 
\begin{equation}
\label{topocosmico}
H^i_{dR}(M)\cong \ker(\Delta_i)\cong \ker(\Delta_{m-i})\cong H^{m-i}_{dR}(M) 
\end{equation} 
where $\Delta_i:=d_{i-1}\circ \delta_{i-1}+\delta_i\circ d_i$ is the $i-$th Hodge Laplacian acting on $\Omega^i(M)$. 
As it is well known \eqref{topobirbone} and \eqref{topocosmico} are not longer true when $M$ is not compact.\\ In this case two natural and important variations of the de Rham cohomology  are provided by the $L^2$-de Rham cohomology and by the reduced $L^2$-de Rham cohomology. We recall briefly  that the reduced maximal $L^2$-cohomology, $\overline{H}^i_{2,max}(M,g)$, is defined as $\ker(d_{max,i})/\overline{\im(d_{max,i-1})}$ while the maximal $L^2$-cohomology, $H^i_{2,max}(M,g)$, is defined as $\ker(d_{max,i})/\im(d_{max,i-1})$ where $d_{max,i}:L^2\Omega^i(M,g)\rightarrow L^2\Omega^{i+1}(M,g)$ is the distributional extension of $d_i:\Omega^i_c(M)\rightarrow \Omega^i_c(M)$. Analogously the reduced minimal $L^2$-cohomology, $\overline{H}^i_{2,min}(M,g)$, is defined taking the quotient $\ker(d_{min,i})/\overline{\im(d_{min,i-1})}$ while the minimal $L^2$-cohomology, $H^i_{2,min}(M,g)$, is defined as $\ker(d_{min,i})/\im(d_{min,i-1})$ where $d_{min,i}:L^2\Omega^i(M,g)\rightarrow L^2\Omega^{i+1}(M,g)$ is the graph closure of $d_i:\Omega^i_c(M)\rightarrow \Omega^i_c(M)$. In the non compact setting they are an important tool and indeed they have been the subject of many studies during the last decades. In this case, as it is well known, the completeness of $(M,g)$ plays a fundamental role. When $(M,g)$ is complete, the Laplacian $\Delta_i$, with domain given by the smooth and compactly supported forms $\Omega_c^i(M)$, is an essentially self-adjoint operator on $L^2\Omega^i(M,g)$. In particular this implies that Poincar\'e duality holds for the reduced $L^2$-cohomology of $(M,g)$. Therefore, when the $L^2$-cohomology is finite dimensional,  it coincides with the  reduced $L^2$-cohomology and so it satisfies Poincar\'e duality.   All these properties in general fail when $(M,g)$ is incomplete.  Generally in this case the differential $d_i$ acting on smooth $i-$forms with compact support admits several different closed extensions when we look at it as an unbounded operator  between $L^2\Omega^i(M,g)$ and $L^2\Omega^{i+1}(M,g)$. Therefore, depending on the closed extensions considered, we will get different  $L^2$-cohomology groups and $L^2$-reduced cohomology groups for which, in general, Poincar\'e duality does not hold. However open and incomplete riemannian manifolds appear naturally in the context of riemannian geometry and in that of  global analysis, in particular when we deal with spaces with "singularities" such as stratified pseudomanifolds or singular complex (or real) algebraic varieties. Therefore it is an interesting question  to investigate some general constructions for the $L^2$-cohomology of $(M,g)$, when $g$ is incomplete, such that  suitable versions of \eqref{topobirbone} and \eqref{topocosmico} are satisfied or, briefly,  such that Poincar\'e duality holds. In the literature other papers have dealt with this question: for example we mention \cite{ALMP}, \cite{FBE} and  \cite{BL}.\\This paper is organized in the following way:  The first chapter is devoted to  Hilbert complexes. As explained by Br\"uning and Lesch in \cite{BL} this is the natural framework to describe the general properties of an elliptic complex from an $L^2$ point of view. We start recalling from \cite{BL} the main properties and definitions and then we prove some abstract results about Poincar\'e duality for Hilbert complexes. \\In the second section, after recalled the notion of $L^2$-de Rham cohomology, we apply the results of the first chapter to the case of the $L^2$-de Rham complex.   We can summarize our main results in the following way:
\begin{teo}
\label{marioq}
Let $(M,g)$ be an open, oriented and incomplete riemannian manifold of dimension $m$. Then, for each $i=0,...,m$, we have the following isomorphism: $$\ker(d_{max,i})/\overline{\im(d_{min,{i-1}})}\cong \ker(d_{max,m-i})/\overline{\im(d_{min,{m-i-1}})}.$$  \\Assume  now that, for each $i=0,...,m$, $\im(d_{min,i})$ is closed in $L^2\Omega^{i+1}(M,g)$. Then  there exists a Hilbert complex $(L^2\Omega^i(M,g)),d_{\mathfrak{M},i})$ which satisfies the following properties for each $i=0,...,m$:  
\begin{itemize}
\item $\mathcal{D}(d_{min,i})\subseteq \mathcal{D}(d_{\mathfrak{M},i})\subseteq \mathcal{D}(d_{max,i}),$ that is $d_{max,i}$ is an extension of $d_{\mathfrak{M},i}$ which is an extension of $d_{min,i}$.  
\item $\im(d_{\mathfrak{M},i})$ is closed in $L^2\Omega^{i+1}(M,g)$.    
\item  If we call $H^i_{2,\mathfrak{M}}(M,g)$  the cohomology of the Hilbert complex $(L^2\Omega^i(M,g),d_{\mathfrak{M},i})$ then we have: $$H^i_{2,\mathfrak{M}}(M,g)=\ker(d_{max,i})/\im(d_{min,i})$$ and $$H^i_{2,\mathfrak{M}}(M,g)\cong H^{m-i}_{2,\mathfrak{M}}(M,g).$$
\item There exists a  well defined and non degenerate pairing: $$H^i_{2,\mathfrak{M}}(M,g)\times H^{m-i}_{2,\mathfrak{M}}(M,g)\longrightarrow \mathbb{R},\ ([\omega],[\eta])\longmapsto \int_{M}\omega\wedge \eta.$$
\end{itemize}
\end{teo}

Then  we prove that:

\begin{teo}
\label{dualcomplexformz}
Under the assumptions of Theorem \ref{marioq}.  Consider the Hilbert complexes: 
\begin{equation}
\label{topoinfrarossoz}
0\leftarrow L^2(M,g)\stackrel{\delta_{max,0}}{\leftarrow}L^2\Omega^1(M,g)\stackrel{\delta_{max,1}}{\leftarrow}L^2\Omega^{2}(M,g)\stackrel{\delta_{max,2}}{\leftarrow}...\stackrel{\delta_{max,n-1}}{\leftarrow}L^2\Omega^{n}(M,g)\leftarrow 0,
\end{equation}
and 
\begin{equation}
\label{topodellozioz}
0\leftarrow L^2(M,g)\stackrel{\delta_{min,0}}{\leftarrow}L^2\Omega^1(M,g)\stackrel{\delta_{min,1}}{\leftarrow}L^2\Omega^{2}(M,g)\stackrel{\delta_{min,2}}{\leftarrow}...\stackrel{\delta_{min,n-1}}{\leftarrow}L^2\Omega^{n}(M,g)\leftarrow 0
\end{equation}
Let 
\begin{equation}
\label{topoultraviolettoz}
0\leftarrow L^2(M,g)\stackrel{\delta_{\mathfrak{M},0}}{\leftarrow}L^2\Omega^1(M,g)\stackrel{\delta_{\mathfrak{M},1}}{\leftarrow}L^2\Omega^{2}(M,g)\stackrel{\delta_{\mathfrak{M},2}}{\leftarrow}...\stackrel{\delta_{\mathfrak{M},n-1}}{\leftarrow}L^2\Omega^{n}(M,g)\leftarrow 0
\end{equation}
be the intermediate complex, which extends \eqref{topodellozioz} and which is extended by \eqref{topoinfrarossoz}, built according to Theorem \ref{marioq}. Then, for each $i=0,...,m$,  we have:
\begin{equation}
\label{topopressatoz}
d_{\mathfrak{M},i}^*=\delta_{\mathfrak{M},i}=\pm *d_{\mathfrak{M},i}*
\end{equation}
where $d_{\mathfrak{M},i}^*$ is the adjoint of $d_{\mathfrak{M},i}$ and $*$ is the Hodge star operator.
\end{teo}

Moreover we prove the following result:

\begin{teo}
\label{nuvolosoz}
Let $(M,g)$ be an open, oriented and incomplete riemannian manifold of dimension $m$. Suppose that, for each $i=0,...,m$, $\im(d_{min,i})$ is closed in $L^2\Omega^{i+1}(M,g)$. Let $(L^2\Omega^i(M,g),d_{\mathfrak{M},i})$ be the Hilbert complex built in Theorem \ref{marioq}. Assume that $(L^2\Omega^i(M,g),d_{\mathfrak{M},i})$ is a Fredholm complex. Then:
\begin{enumerate}
\item  Every  closed extension $(L^2\Omega^i(M,g),D_i)$ of $(\Omega_c^i(M),d_i)$ is a Fredholm complex.
\item For every $i=0,...,m$  the quotient of the domain of $d_{max,i}$ with the domain of $d_{min,i}$, that is $$\mathcal{D}(d_{max,i})/\mathcal{D}(d_{min,i})$$ is a finite dimensional vector space.
\end{enumerate}
\end{teo}

According to Theorem \ref{nuvolosoz}, we  define the following number associated to $(M,g)$:
\begin{equation}
\label{boiachiz}
\psi_{L^2}(M,g):=\sum_{i=0}^m(-1)^i\dim(\mathcal{D}(d_{max,i})/\mathcal{D}(d_{min,i}))
\end{equation}
and we prove  the following formula:
\begin{teo}
\label{zarroz}
Under the  hypotheses of Theorem \ref{nuvolosoz}. The following formula holds:
\begin{equation}
\label{merisiz}
\psi_{L^2}(M,g)=\chi_{2,M}(M,g)-\chi_{2,m}(M,g)=\left\{
\begin{array}{ll}
0\ &\ \dim(M)\ is\ even\\
2\chi_{2,M}(M,g)\ &\ \dim(M)\ is\ odd
\end{array}
\right.
\end{equation}
where $\chi_{2,M}(M,g)$ and $\chi_{2,m}(M,g)$ are the Euler characteristics associated respectively to the complexes $(L^2\Omega^i(M,g),d_{max,i})$ and $(L^2\Omega^i(M,g),d_{min,i})$.
\end{teo}
In the remaining part of the second chapter and in the third one we prove other results for the complex $(L^2\Omega^i(M,g),d_{\mathfrak{M},i})$. In particular we prove a Hodge type theorem  for the cohomology groups $H^i_{2,\mathfrak{M}}(M,g)$, we introduce the  $L^2-$Euler characteristic $\chi_{2,\mathfrak{M}}(M,g)$ associated to $(L^2\Omega^i(M,g),d_{\mathfrak{M},i})$ and the $L^2-$signature $\sigma_{2,\mathfrak{M}}(M,g)$ for $(M,g)$ when $\dim(M)=4l$. Then we show that they are the index of some suitable Fredholm operators arising from the complex $(L^2\Omega^i(M,g),d_{\mathfrak{M},i})$. Finally the last part of the paper contains some examples and applications of the previous results.\\ We conclude this introduction mentioning that in a subsequent paper we plan to come back again on this subject investigating some topological properties of the vector spaces $H^i_{2,\mathfrak{M}}(M,g)$ with particular attention to the cases when they are finite dimensional. \vspace{1 cm}

\textbf{Acknowledgments.} I wish to thank  Jochen Br\"uning for many interesting discussions and helpful hints. I also wish to thank Pierre Albin,  Erich Leichtnam, Rafe Mazzeo and Paolo Piazza for interesting comments and emails.  This research has been financially supported by the SFB 647 : Raum-Zeit-Materie.

\section {Hilbert Complexes}

We start the section recalling the notion of Hilbert complex and its main properties.  For a complete development of the subject we refer to \cite{BL}.

\begin{defi} 
A Hilbert complex is a complex, $(H_{*},D_{*})$ of the form:
\begin{equation}
\label{mm}
0\rightarrow H_{0}\stackrel{D_{0}}{\rightarrow}H_{1}\stackrel{D_{1}}{\rightarrow}H_{2}\stackrel{D_{2}}{\rightarrow}...\stackrel{D_{n-1}}{\rightarrow}H_{n}\rightarrow 0,
\end{equation}
where each $H_{i}$ is a separable Hilbert space and each map $D_{i}$ is a closed operator called the differential such that:
\begin{enumerate}
\item $\mathcal{D}(D_{i})$, the domain of $D_{i}$, is dense in $H_{i}$.
\item $\im(D_{i})\subset \mathcal{D}(D_{i+1})$.
\item $D_{i+1}\circ D_{i}=0$ for all $i$.
\end{enumerate}
\end{defi}
The cohomology groups of the complex are $H^{i}(H_{*},D_{*}):=\ker(D_{i})/\im(D_{i-1})$. If the groups $H^{i}(H_{*},D_{*})$ are all finite dimensional we say that the complex is a  $Fredholm\ complex$. 

Given a Hilbert complex there is a dual Hilbert complex
\begin{equation}
0\leftarrow H_{0}\stackrel{D_{0}^{*}}{\leftarrow}H_{1}\stackrel{D_{1}^{*}}{\leftarrow}H_{2}\stackrel{D_{2}^{*}}{\leftarrow}...\stackrel{D_{n-1}^{*}}{\leftarrow}H_{n}\leftarrow 0,
\label{mmp}
\end{equation}
defined using $D_{i}^{*}:H_{i+1}\rightarrow H_{i}$, the Hilbert space adjoint of the differential $D_{i}:H_{i}\rightarrow H_{i+1}$. The cohomology groups of $(H_{j},(D_{j})^*)$, the dual Hilbert  complex, are $$H^{i}(H_{j},(D_{j})^*):=\ker(D_{n-i-1}^{*})/\im(D_{n-i}^*).$$\\ An important self-adjoint operator associated to \eqref{mm} is the following one: let us label $H:=\bigoplus_{i=0}^nH_i$ and let 
\begin{equation}
\label{aster}
D+D^*:H\rightarrow H
\end{equation}
be the self-adjoint operator with domain
$$\mathcal{D}(D+D^*)=\bigoplus_{i=0}^n (\mathcal{D}(D_{i})\cap \mathcal{D}(D^*_{i-1}))$$
and defined as $$D+D^*:=\bigoplus_{i=0}^n(D_i+D^*_{i-1}).$$
Moreover, for all $i$, there is also a Laplacian $\Delta_{i}=D_{i}^{*}D_{i}+D_{i-1}D_{i-1}^{*}$ which is a self-adjoint operator on $H_{i}$ with domain 
\begin{equation}
\label{saed}
\mathcal{D}(\Delta_{i})=\{v\in \mathcal{D}(D_{i})\cap \mathcal{D}(D_{i-1}^{*}): D_{i}v\in \mathcal{D}(D_{i}^{*}), D_{i-1}^{*}v\in \mathcal{D}(D_{i-1})\}
\end{equation} and nullspace: 
\begin{equation}
\label{said}
\mathcal{H}^{i}(H_{*},D_{*}):=\ker(\Delta_{i})=\ker(D_{i})\cap \ker(D_{i-1}^{*}).
\end{equation}

The following propositions are well known. The first result is the weak Kodaira decomposition:

\begin{prop}
\label{beibei}
 [\cite{BL}, Lemma 2.1] Let $(H_{i},D_{i})$ be a Hilbert complex and $(H_{i},(D_{i})^{*})$ its dual complex, then: \begin{equation}
\label{kudam}
H_{i}=\mathcal{H}^{i}\oplus\overline{\im(D_{i-1})}\oplus\overline{\im(D_{i}^{*})}.
\end{equation}
\end{prop}

The reduced cohomology groups of the complex are: $$\overline{H}^{i}(H_{*},D_{*}):=\ker(D_{i})/(\overline{\im(D_{i-1})}).$$
By the above proposition there is a pair of  weak de Rham isomorphism theorems:
 \begin{equation}
\label{pppp}
\left\{
\begin{array}{ll}
\mathcal{H}^{i}(H_{j},D_{j})\cong\overline{H}^{i}(H_{j},D_{j})\\
\mathcal{H}^{i}(H_{j},D_{j})\cong\overline{H}^{n-i}(H_{j},(D_{j})^{*})
\end{array}
\right.
\end{equation}
where in the second case we mean the cohomology of the dual Hilbert complex.\\
The complex $(H_{*},D_{*})$ is called $ weakly\  Fredholm$  if $\mathcal{H}^{i}(H_{*},D_{*})$ is finite dimensional for each $i$. By the next propositions we get immediately that each Fredholm complex is a weak Fredholm complex.

\begin{prop}
\label{topoamotore}
[\cite{BL}, corollary 2.5] If the cohomology of a Hilbert complex $(H_{*}, D_{*})$ is finite dimensional then, for all $i$,  $\im(D_{i-1})$ is closed  and  $H^{i}(H_{*},D_{*})\cong \mathcal{H}^{i}(H_{*},D_{*}).$
\end{prop}

\begin{prop}
\label{worref}
The following properties are equivalent:
\begin{enumerate}
\item \eqref{mm} is a Fredholm complex.
\item The operator defined in \eqref{aster} is a Fredholm operator on its domain endowed with the graph norm.
\item For all $i=0,...,n$ $\Delta_i:\mathcal{D}(\Delta_i)\rightarrow H_i$ is a Fredholm operator on its domain endowed with the graph norm.
\end{enumerate}
\end{prop}

\begin{proof}
See \cite{BL} Theorem 2.4
\end{proof}

\begin{prop}
\label{fred}
[\cite{BL}, corollary 2.6] A Hilbert complex  $(H_{j},D_{j}),\ j=0,...,n$ is a Fredholm complex (weakly Fredholm) if and only if  its dual complex, $(H_{j},D_{j}^{*})$, is Fredholm (weakly Fredholm). In the Fredholm case we have:
\begin{equation}
\mathcal{H}^{i}(H_{j},D_{j})\cong H^{i}(H_{j},D_{j})\cong H^{n-i}(H_{j},(D_{j})^{*})\cong \mathcal{H}^{n-i}(H_{j},(D_{j})^{*}).
\end{equation}
Analogously in the  weak Fredholm case we have:
\begin{equation}
\mathcal{H}^{i}(H_{j},D_{j})\cong\overline{ H}^{i}(H_{j},D_{j})\cong\overline{ H}^{n-i}(H_{j},(D_{j})^{*})\cong \mathcal{H}^{n-i}(H_{j},(D_{j})^{*}).
\end{equation}
\end{prop}
Now we recall some definitions from \cite{FBE}. We refer to the same paper for more properties and comments.
\begin{defi} 
\label{edera}
Consider  a pair of Hilbert complexes $(H_{i},D_{i})$ and $(H_{i},L_{i})$ with $i=0,...,n$. The pair  $(H_{i},D_{i})$ and $(H_{i},L_{i})$ is said to be \textbf{ complementary} if the following property is satisfied:
\begin{itemize}
\item for each $i$ there exists an isometry $\phi_{i}:H_{i}\longrightarrow H_{n-i}$ such that $\phi_i(\mathcal{D}(D_i))=\mathcal{D}(L^*_{n-i-1})$ and  $L_{n-i-1}^*\circ\phi_{i}=C_{i}(\phi_{i+1}\circ D_{i})$ on $\mathcal{D}(D_{i})$
 where $ L^{*}_{n-i-1}:H_{n-i}\longrightarrow H_{n-i-1}$ is the adjoint of $L_{n-i-1}:H_{n-i-1}\longrightarrow H_{n-i}$ 
 and $C_{i}\neq 0$ is a constant which depends only on $i$.
\end{itemize}
We call the maps $\phi_{i}$  duality maps.
\end{defi}

We have the following proposition:
\begin{prop}
\label{errew}
Let $(H_{i},D_{i})$ and $(H_{i},L_{i})$ be complementary Hilbert complexes. Then:
\begin{enumerate}
\item Also  $(H_{i},L_{i})$ and  $(H_{i},D_{i})$ are complementary  Hilbert complexes. Moreover if $\{\phi_{i}\}$ are the duality maps which make $(H_{i},D_{i})$ and $(H_{i},L_{i})$ complementary  then  $\{\phi_{i}^{*}\}$, the family obtained taking the adjoint maps,  are the duality maps which make $(H_{i},L_{i})$ and $(H_{i},D_{i})$  complementary.
\item Each $\phi_{j}$ induces an isomorphism between $\mathcal{H}^{j}(H_{*},D_{*})$ and $\mathcal{H}^{n-j}(H_{*},L_{*})$.
\item The complexes $(H_{i},D_{i})$ and $(H_{i},L_{i}^{*})$ have isomorphic cohomology groups and isomorphic reduced cohomology groups. In the same way the complexes $(H_{i},L_{i})$ and $(H_{i},D_{i}^{*})$ have isomorphic cohomology groups and isomorphic reduced cohomology groups.
\item The following isomorphism holds: $\overline{H}^{j}(H_{*},D_{*})\cong \overline{H}^{n-j}(H_{*},L_{*}).$
\end{enumerate}
\end{prop}

\begin{proof}
See \cite{FBE} Prop. 5.
\end{proof}

Finally, given a pair of Hilbert complexes $(H_{j},D_{j})$ and $(H_{j},D'_{j})$, we will write $(H_{j},D_{j})\subseteq  (H_{j},D'_{j}) $ if, for each $j$,  $D'_j$ extends $D_j$. We will write $(H_{j},D_{j})\subset  (H_{j},D'_{j})$ if $D_j\neq D_j'$ for at least one $j$. 
We are now in position to prove the main results of this section:

\begin{teo}
\label{berlino}
Let $(H_{j},D_{j})\subseteq (H_{j},L_{j})$ be  a pair of complementary Hilbert complexes. Then, for every $j=0,...,n$, we have the following isomorphism:  
\begin{equation}
\ker(L_{j})/(\overline{\im(D_{j-1})})\cong \ker(L_{n-j})/(\overline{\im(D_{n-j-1})}).
\label{kiol}
\end{equation}
\end{teo}

\begin{proof}
The fact that $L_j$ is an extension of $D_j$ implies that, the complex below is well defined  for each  $j=0,...,n$
\begin{equation}
0\rightarrow H_{0}\stackrel{D_{0}}{\rightarrow}H_{1}\stackrel{D_{1}}{\rightarrow}...\stackrel{D_{j-1}}{\rightarrow}H_{j}\stackrel{L_{j}}{\rightarrow}...\stackrel{L_{n-1}}{\rightarrow}H_{n}\rightarrow 0.
\label{mim}
\end{equation}
The dual Hilbert complex is clearly:
\begin{equation}
0\leftarrow H_{0}\stackrel{D_{0}^{*}}{\leftarrow}...\stackrel{D_{j-1}^{*}}{\leftarrow}H_{j}\stackrel{L_{j}^{*}}{\leftarrow}...\stackrel{L_{n-1}^{*}}{\leftarrow}H_{n}\leftarrow 0,
\label{mmpp}
\end{equation}
Therefore, by \eqref{said} and \eqref{pppp}, we get that:
\begin{equation}
\label{samaria}
\ker(L_j)/(\overline{\im(D_{j-1})})\cong \ker(L_j)\cap \ker(D_{j-1}^*).
\end{equation} By Def. \ref{edera} and Prop. \ref{errew} we know that $\phi_{n-j}$ induces an isomorphism between $\ker(L_j)$ and $\ker(D^*_{n-j-1})$ and between $\ker(L_{n-j})$ and $\ker (D^*_{j-1})$. Therefore it induces an isomorphism between  $\ker(L_j)\cap \ker(D^*_{j-1})$ and $\ker(D^*_{n-j-1})\cap \ker(L_{n-j})$. In this way, using \eqref{samaria}, we get:
$$\ker(L_j)/(\overline{\im(D_{j-1})})\cong \ker(L_j)\cap \ker(D_{j-1}^*)\cong$$ $$\cong \ker(D^*_{n-j-1})\cap \ker(L_{n-j})\cong \ker(L_{n-j})/(\overline{\im(D_{n-j-1})})$$  and this completes the proof.
\end{proof}

\begin{teo}
\label{cheafa}
Let  $(H_{j},D_{j})\subseteq (H_{j},L_{j})$,  $j=0,...,n$, be a pair of  Hilbert complexes. Suppose that, for each $j$, $\im(D_{j})$ is closed in $H_{j+1}$. Then there exists a third Hilbert complex $(H_{j},P_{j})$ such that: 
\begin{enumerate}
\item  $(H_{j},D_{j})\subseteq (H_{j},P_{j})\subseteq (H_{j},L_{j})$ and the image of  $P_j$ is closed for each $j$.
\item $H^j(H_{*},P_*)=\ker(L_{j})/(\im(D_{j-1}))$. 
\item  If $(H_{j},D_{j})\subseteq (H_{j},L_{j})$ are complementary then: $$H^j(H_{*},P_{*})\cong H^{n-j}(H_{*},P_{*}).$$
\end{enumerate}
\end{teo}

\begin{proof}
To prove the first part of the proposition we have to exhibit a Hilbert complex which satisfies the assertions of the statement. To do this consider the following Hilbert space $$(\mathcal{D}(L_j),\langle\ ,\ \rangle_{\mathcal{G}})$$ which is by definition  the domain of $L_j$ endowed with the graph scalar product, that is for each pair of  elements $u, v\in \mathcal{D}(L_j)$ we have $$\langle u,v\rangle_{\mathcal{G}}:=\langle u,v\rangle_{H_j}+\langle L_ju,L_jv\rangle_{H_{j+1}}.$$ During the rest of the proof we will work with this Hilbert space and therefore  all the direct sums that will appear  and all the assertions of topological type are referred to this Hilbert space $(\mathcal{D}(L_j),\langle\ ,\ \rangle_{\mathcal{G}})$. We can decompose $(\mathcal{D}(L_j),\langle\ ,\ \rangle_{\mathcal{G}})$ in the following way:
\begin{equation}
\label{wsws}
(\mathcal{D}(L_j),\langle\ ,\ \rangle_{\mathcal{G}})=\ker(L_j)\oplus V_j
\end{equation}
where $V_j=\{\alpha\in \mathcal{D}(L_j)\cap \overline{\im(L_{j}^*)}\}$ and it is immediate to check that these subspaces are both closed  in $(\mathcal{D}(L_j),\langle\ ,\ \rangle_{\mathcal{G}})$.\\Consider now $(\mathcal{D}(D_j),\langle\ ,\ \rangle_{\mathcal{G}})$; it is a closed subspace of $(\mathcal{D}(L_j),\langle\ , \rangle_{\mathcal{G}})$ and we can decompose it as 
\begin{equation}
\label{wswsa}
(\mathcal{D}(D_j),\langle\ ,\ \rangle_{\mathcal{G}})=\ker(D_j)\oplus A_j. 
\end{equation}
By the assumption on the range of $D_j$ we get that also the range of $D_{j}^*$ is closed. So,  analogously to the previous case,  $A_j=\{\alpha\in \mathcal{D}(D_j)\cap \im(D_{j}^*)\}$ and obviously these subspaces are both closed  in $(\mathcal{D}(D_j),\langle\ ,\ \rangle_{\mathcal{G}})$. 
Clearly if $\ker(D_j)=\ker(L_j)$ then the Hilbert complex $(H_j,D_j)$ satisfies the first two properties of the statement, that is defining $(H_j,P_j)$ as $(H_j,D_j)$, we have $(H_{j},D_{j})\subseteq (H_{j},P_{j})\subseteq (H_{j},L_{j})$, the image of  $P_j$ is closed for each $j$ and $H^j(H_{*},P_*)=\ker(L_{j})/(\im(D_{j-1}))$. So we can suppose that $\ker(D_j)$ is properly contained in $\ker(L_j)$. Let $\pi_{1,j}$ be the orthogonal projection of $A_j$ onto $\ker(L_j)$ and analogously let $\pi_{2,j}$ be the orthogonal projection of $A_j$ onto $V_j$. We have the following properties:
\begin{enumerate}
\item $\pi_{2,j}$ is injective 
\item $\im(\pi_{2,j})$ is closed.
\end{enumerate}
The first property follows from the fact that  $\ker(\pi_{2,j})=A_j\cap \ker(L_j)$. But $L_j$ is an extension of $D_j$; therefore if an element $\alpha$ lies in  $A_j\cap \ker(L_j)$ then it lies also in $\ker(D_j)$ and so $\alpha=0$ because $\ker(D_j)\cap A_j=\{0\}$.  For the second property consider a sequence $\{\gamma_{m}\}_{m\in \mathbb{N}}\subset A_j$ such that $\pi_{2,j}(\gamma_{m})$ converges to $\gamma\in V_j$. We recall that we are in $(\mathcal{D}(L_j),\langle\ ,\ \rangle_{\mathcal{G}})$ and therefore this means that $$\lim_{m\rightarrow \infty}\pi_{2,j}(\gamma_m)= \gamma\ \text{in}\ H_j\ \text{and}\ \lim_{m\rightarrow\infty}L_j(\pi_{2,j}(\gamma_m))= L_j(\gamma)\ \text{in}\ H_{j+1}.$$ Then $$\lim_{m\rightarrow \infty}D_{j}(\gamma_m)=\lim_{m\rightarrow \infty}L_{j}(\gamma_m)=\lim_{m\rightarrow \infty}L_{j}(\pi_{2,j}(\gamma_m))=L_j(\gamma).$$ This implies that $$\lim_{m\rightarrow \infty}D_{j}(\gamma_m)=L_j(\gamma)$$ and therefore the limit  exists. So by the assumptions about the range of $D_j$ we get that there exists an element $\eta\in A_j$ such that $$\lim_{m\rightarrow \infty}D_{j}(\gamma_m)=D_{j}(\eta).$$Moreover $L_{j}(\gamma)=D_{j}(\eta)=L_{j}(\eta)=L_{j}(\pi_{2,j}(\eta))$. This implies that $L_j(\pi_{2,j}(\eta)-\gamma)=0$ and therefore $\pi_{2,j}(\eta)=\gamma$ because $\pi_{2,j}(\eta),\gamma \in V_j$ and $L_i$ is injective on $V_j$. In this way we have shown that $\im(\pi_{2,j})$ is closed.\\
Now define $N_j$ as the range of $\pi_{2,j}$. Finally define $W_j$ as the vector space generated by the sum of $\ker(L_j)$ and $N_j$. By the fact that $\ker(L_j)$ and $N_j$ are orthogonal to each other we have $W_j= \ker(L_j)\oplus N_j$ and therefore $W_j$ is closed in $ (\mathcal{D}(L_j), \langle\ ,\ \rangle_{\mathcal{G}})$. Finally define $P_j$ as 
\begin{equation}
\label{toposorce}
P_j:=L_j|_{W_j}
\end{equation}  
By the fact that  $W_j$ is closed in $\mathcal{D}(L_j)$  and that $\pi_{1,j}(A_j),\pi_{2,j}(A_j)\subset W_j$ we get that $P_j$ is a closed extension of $D_j$ which is in turn  extended by $L_j$. Moreover, by the construction, it is clear that $\ker(P_j)=\ker(L_j)$. Finally, again by the definition of $P_j$ and its domain, we have $\im(P_j)=L_j(\pi_{2,j}(A_j))=\im(D_j)$. Therefore we got that $\im(P_j)$ is closed and that $$\ker(P_j)/\im(P_{j-1})=\ker(L_j)/\im{D_{j-1}}.$$ This completes the proof of the first two statements.\\ Finally, combining the second statement of this Theorem with Theorem \ref{berlino}, the third statement follows. 
\end{proof}

For the dual complex of $(H_i,P_i)$ we have the following description:
\begin{teo}
\label{dualcomplex}
Under the hypotheses of Theorem \ref{cheafa}. Assume moreover that the image of $L_i$, $\im(L_i)$,  is closed for each $i=0,...,n$. Consider the Hilbert complexes: 
\begin{equation}
\label{topofulmine}
0\leftarrow H_{0}\stackrel{D_{0}^{*}}{\leftarrow}H_{1}\stackrel{D_{1}^{*}}{\leftarrow}H_{2}\stackrel{D_{2}^{*}}{\leftarrow}...\stackrel{D_{n-1}^{*}}{\leftarrow}H_{n}\leftarrow 0,
\end{equation}
and 
\begin{equation}
\label{topoenergia}
0\leftarrow H_{0}\stackrel{L_{0}^{*}}{\leftarrow}H_{1}\stackrel{L_{1}^{*}}{\leftarrow}H_{2}\stackrel{L_{2}^{*}}{\leftarrow}...\stackrel{L_{n-1}^{*}}{\leftarrow}H_{n}\leftarrow 0.
\end{equation}
Let 
\begin{equation}
\label{topoenergiapura}
0\leftarrow H_{0}\stackrel{S_{0}}{\leftarrow}H_{1}\stackrel{S_{1}}{\leftarrow}H_{2}\stackrel{S_{2}}{\leftarrow}...\stackrel{S_{n-1}}{\leftarrow}H_{n}\leftarrow 0
\end{equation}
be the intermediate complex, which extends \eqref{topoenergia} and which is extended by \eqref{topofulmine}, constructed according to Theorem \ref{cheafa} (and its proof). Then, for each $i=0,...,n$,  we have:
\begin{equation}
\label{topotuono}
P_i^*=S_i.
\end{equation}
Furthermore assume that $(H_i,D_i)$ and $(H_i,L_i)$ are complementary (in this case the fact that $\im(D_i)$ is closed implies that $\im(L_i)$ is closed). Let $\{\phi_i\}$ be the duality maps and suppose that $\phi_i^{-1}=\pm \phi_{n-i}$. Then we have: 
\begin{equation}
\label{topocentrifuga}
S_i=C_i^{-1}\phi_{i}^{-1}\circ P_{n-i-1}\circ \phi_{i+1}=\pm C_i^{-1}\phi_{n-i}\circ P_{n-i-1}\circ \phi_{i+1}.
\end{equation}
\end{teo}

In order to prove Theorem \ref{dualcomplex} we need the following proposition:
\begin{prop}
\label{topolampo}
Let $H$ and $K$ be two Hilbert spaces and let $T:H\rightarrow K$ be a closed and densely defined  operator. Let $S:H\rightarrow K$ be another closed and densely defined operator which extends $T$. Assume that $\ker(T)=\ker(S)$ and $\im(T)=\im(S)$. Then: $$T=S$$ that is $\mathcal{D}(T)=\mathcal{D}(S)$ and $T(u)=S(u)$ for each $u\in \mathcal{D}(S)$.
\end{prop}

\begin{proof}
Consider the Hilbert space $(\mathcal{D}(S),\langle\ ,\ \rangle_{\mathcal{G}})$. As in the proof of Theoren \ref{cheafa} we can decompose it as $(\mathcal{D}(S),\langle\ ,\ \rangle_{\mathcal{G}})=\ker(S)\oplus A$ where $A=\mathcal{D}(S)\cap \overline{\im(S^*)}$. Analogously, if we consider $(\mathcal{D}(T),\langle\ ,\ \rangle_{\mathcal{G}})$, then we have $(\mathcal{D}(T),\langle\ ,\ \rangle_{\mathcal{G}})=\ker(T)\oplus B$ where $B=\mathcal{D}(T)\cap \overline{\im(T^*)}$. By the fact that $\mathcal{D}(T)\subseteq \mathcal{D}(S)$ and $\ker(S)=\ker(T)$ we get that $B\subseteq A$. Now let $u\in A$. Then there exists $v\in B$ such that $S(u)=T(v)$. Therefore, by the fact that $S$ extends $T$, we have $S(u-v)=0$ and this implies that $u=v$ because $(u-v)\in \ker(S)\cap A$. So we can conclude that $S=T$.
\end{proof}

\begin{proof} (of Theorem \ref{dualcomplex}).
First of all we remark that we can apply Theorem \ref{cheafa} to the pair of complexes \eqref{topofulmine} and \eqref{topoenergia}. Clearly  \eqref{topofulmine} extends \eqref{topoenergia}; moreover, having assumed that $\im(L_i)$ is closed, it follows that $\im(L_i^*)$ is closed. In this way the assumptions of Theorem \ref{cheafa} are fulfilled. Now, by Theorem \ref{cheafa} and its proof, we know that $\im(P_i)$ is closed for each $i=0,...,n$. Therefore also $\im(P_i^*)$ is closed for each $i=0,...,n$ and we have $\im(P_i^*)=(\ker(P_i))^{\bot}=$ $(\ker(L_i))^{\bot}=\im(L^*_i)$. In the same way $\ker(P_i^*)=$ $(\im(P_i))^{\bot}=(\im(D_i))^{\bot}=$ $\ker(D^*_i)$. Now, if we consider $S_i$, again by Theorem \ref{cheafa} and its proof, we have $\im(S_i)=\im(L_i^*)$, $\ker(S_i)=\ker(D_i^*)$ and in particular $\im(S_i)$ is closed. Therefore, according to Prop. \ref{topolampo}, in order to prove \eqref{topotuono} it is enough  to show that $P_i^*$ extends $S_i$. To do this we have to show that:
\begin{equation}
\label{conditio}
\langle P_i(u),v\rangle_{H_{i+1}} =\langle u,S_i(v)\rangle_{H_i}
\end{equation}
for each $u\in \mathcal{D}(P_i)$ and for each $v\in \mathcal{D}(S_i).$ We start observing that we can decompose $u$ as $u_1+u_2$ where $u_1\in \ker(P_i)$ and $u_2\in \mathcal{D}(P_i)\cap \im(P^*_i)$. Analogously $v=v_1+v_2$ where $v_1\in \ker(S_i)$ and $v_2\in \mathcal{D}(S_i)\cap \im(S^*_i)$. So we get $\langle P_i(u),v\rangle_{H_{i+1}} =\langle P_i(u_2),v_2\rangle_{H_{i+1}}$ because $0=P_i(u_1)$ and $P_i(u_2)\in \im(D_i)$  which is orthogonal to $\ker(D_i^*)=\ker(S_i)$. By the proof of Theorem \ref{cheafa} we know that $P_i(u_2)=D_i(w)$ for a unique element $w\in \mathcal{D}(D_i)\cap \im(D_i^*)$. Therefore $\langle P_i(u_2),v_2\rangle_{H_{i+1}}$ $=\langle D_i(w),v_2\rangle_{H_{i+1}}$ $=\langle w,D_i^*(v_2)\rangle_{H_i}$ $=\langle w,S_i(v_2)\rangle_{H_i}$ because $v_2\in \mathcal{D}(S_i)\subset \mathcal{D}(D_i^*)$ and $D_i^*|_{\mathcal{D}(S_i)}=S_i$. Now, using the projections $\pi_{1,i}$ and $\pi_{2,i}$ defined in the proof of Theorem \ref{cheafa} we can decompose $w$ as $\pi_{1,i}(w)+\pi_{2,i}(w)$ where  $\pi_{1,i}(w)$ is the projection of $w$ on $\ker(L_i)$ and $\pi_{2,i}(w)$ is the projection of $w$ on $\mathcal{D}(L_i)\cap \im(L_i^*)$.\\ We have that $\pi_{2,i}(w)=u_2$ because $\pi_{2,i}(w)-u_2\in \ker(L_i)\cap (\mathcal{D}(L_i)\cap \im(L_i^*))$ $=\{0\}$. Therefore we get $\langle w,S_i(v_2)\rangle_{H_i}$ $=\langle \pi_{1,i}(w)+u_2,S_i(v_2)\rangle_{H_i}$. Now by  the fact that $S_i(v_1)=0$ and $\langle z,S_i(v_2)\rangle_{H_i}=0$ for each $z\in \ker(L_i)$ we get  $\langle \pi_{1,i}(w)+u_2,S_i(v_2)\rangle_{H_i}$ $=\langle u_2,S_i (v_2)+S_i(v_1)\rangle_{H_i}$ $=\langle u_1+u_2,S_i(v_2)+S_i(v_1)\rangle_{H_i}$ $=\langle u,S_i (v)\rangle_{H_i}$. Summarizing all the passages we have: $$\langle P_i(u),v\rangle_{H_{i+1}} =\langle P_i(u_2),v_2\rangle_{H_{i+1}}=\langle D_i(w),v_2\rangle_{H_{i+1}}=\langle w,D_i^*(v_2)\rangle_{H_i}=\langle w,S_i(v_2)\rangle_{H_i}=$$ $$=\langle \pi_{1,i}(w)+\pi_{2,i}(w),S_i(v_2)\rangle_{H_i}=\langle \pi_{1,i}(w)+u_2,S_i(v_2)\rangle_{H_i}=\langle u_2,S_i (v_2)+S_i(v_1)\rangle_{H_i}=$$ $$=\langle u_1+u_2,S_i (v_2)+S_i(v_1)\rangle_{H_i}=\langle u,S_i (v)\rangle_{H_i}$$ and this completes the proof of \eqref{topotuono}.\\Now we prove \eqref{topocentrifuga}.  We start recalling  that 
\begin{equation}
\label{topotroppo}
\mathcal{D}(S_i)=\ker(D_i^*)\oplus \pi_{2,i+1}(\mathcal{D}(L^*_i)\cap \im(L_i))
\end{equation}
 and  
\begin{equation}
\label{topoazzurro}
\mathcal{D}(P_{n-i-1})=\ker(L_{n-i-1})\oplus \pi_{2,n-i-1}(\mathcal{D}(D_{n-i-1})\cap \im(D^*_{n-i-1}))
\end{equation}
where, as defined in the proof of Theorem \ref{cheafa},  in \eqref{topotroppo} $\pi_{2,i+1}$ is the projection 
\begin{equation}
\label{topodino}
\pi_{2,i+1}: \mathcal{D}(L^*_i)\cap \im(L_i)\longrightarrow \mathcal{D}(D^*_i)\cap \im(D_i)
\end{equation}
 and in \eqref{topoazzurro}   $\pi_{2,n-i-1}$ is the projection 
\begin{equation}
\label{toposauro}
\pi_{2,n-i-1}:\mathcal{D}(D_{n-i-1})\cap \im(D^*_{n-i-1}) \longrightarrow \mathcal{D}(L_{n-i-1})\cap \im(L^*_{n-i-1}).
\end{equation}
 Therefore, in order to avoid any confusion, during the rest of the proof we will label with $\pi_{2,i+1}^S$ the projection \eqref{topodino} and with $\pi_{2,n-i-1}^P$ the projection \eqref{toposauro}. First of all, in order to establish \eqref{topocentrifuga}, we need to prove  that  $$\phi_{i+1}(\mathcal{D}(S_i))=\mathcal{D}(P_{n-i-1}).$$ This is equivalent to show that 
\begin{equation}
\label{toporazzo}
\phi_{i+1}(\ker(D^*_i))=\ker(L_{n-i-1})
\end{equation}
and that
\begin{equation}
\label{toposaetta}
\phi_{i+1}(\pi_{2,i+1}^S(\mathcal{D}(L^*_i)\cap \im(L_i)))=\pi_{2,n-i-1}^P(\mathcal{D}(D_{n-i-1})\cap \im(D^*_{n-i-1}))
\end{equation}
By the fact that $C_i(\phi_{i+1}\circ D_i)=L_{n-i-1}^*\circ \phi_i$ we get 
\begin{equation}
\label{topoturbina}
C_i(D_i^*\circ \phi_{i+1}^{-1})=\phi_{i}^{-1}\circ L_{n-i-1}
\end{equation}
 and this implies immediately \eqref{toporazzo}.\\Now to establish \eqref{toposaetta} we need to prove  that $\phi_{i+1}(\mathcal{D}(L^*_i)\cap \im(L_i))=\mathcal{D}(D_{n-i-1})\cap \im(D^*_{n-i-1})$ and that $\pi_{2,n-i-1}^P\circ \phi_{i+1}=\phi_{i+1}\circ \pi_{2,i+1}^S$.\\Consider again $C_i(\phi_{i+1}\circ D_i)=L_{n-i-1}^*\circ \phi_i$. It follows immediately that
 \begin{equation}
\phi_i(\mathcal{D}(D_i))=\mathcal{D}(L_{n-i-1}^*). 
\end{equation}
This implies that $\mathcal{D}(D_i)=\phi_{i}^{-1}(\mathcal{D}(L_{n-i-1}^*))$ which in turn implies that $\mathcal{D}(D_i)=\phi_{n-i}(\mathcal{D}(L_{n-i-1}^*))$ or equivalently $\mathcal{D}(D_{n-i-1})=\phi_{i+1}(\mathcal{D}(L_{i}^*))$.\\  Taking again $C_i(\phi_{i+1}\circ D_i)=L_{n-i-1}^*\circ \phi_i$, we get $C_i(D_i^*\circ \phi_{n-i-1})=\pm \phi_{n-i}\circ L_{n-i-1}$ that is $C_{n-i-1}(D_{n-i-1}^*\circ \phi_{i})=\pm \phi_{i+1}\circ L_{i}$. In this way we get that $\phi_{i+1}(\im(L_i))=\im(D_{n-i-1}^*)$. So we can conclude that:  $$\phi_{i+1}(\mathcal{D}(L^*_i)\cap \im(L_i))=\mathcal{D}(D_{n-i-1})\cap \im(D^*_{n-i-1}).$$
Now, to complete the proof of \eqref{toposaetta}, we have to show that $(\phi_{i+1}\circ \pi^S_{2,i+1})(u)=(\pi^P_{2,n-i-1}\circ \phi_{i+1})(u)$ for each $i=0,...,n$ and for each $u\in \mathcal{D}(L_{i}^*)\cap \im(L_i)$.
Let $u\in \mathcal{D}(L_i^*)\cap \im(L_i)$. Then: 
\begin{equation}
\label{topotempesta}
\phi_{i+1}(u)=\pi^P_{1,n-i-1}(\phi_{i+1}(u))+\pi^P_{2,n-i-1}(\phi_{i+1}(u))
\end{equation}
where, as defined in the proof of Theorem \ref{cheafa}, $\pi_{1,n-i-1}^P$ is the projection $$\pi_{1,n-i-1}^P:\mathcal{D}(D_{n-i-1})\cap \im(D^*_{n-i-1}) \longrightarrow \ker (L_{n-i-1}).$$
On the other hand: \begin{equation}
\label{topotornado}
\phi_{i+1}(u)=\phi_{i+1}(\pi^S_{1,i+1}(u)+\pi^S_{2,i+1}(u))
\end{equation}
where now  $\pi_{1,i+1}^S$ is the projection $$\pi_{1,i+1}^S:\mathcal{D}(L^*_i)\cap \im(L_i)\longrightarrow \ker(D_i^*).$$
Now if we look at  \eqref{topoturbina} we get: $$\phi_{i+1}(\mathcal{D}(D^*_{i}))=\mathcal{D}(L_{n-i-1})\ \text{and}\  \phi_{i+1}(\ker(D^*_{i}))=\ker(L_{n-i-1})$$  while  from Def. \ref{edera} we get:  $$\phi_{i+1}(\im(D_{i}))=\im(L^*_{n-i-1}).$$ Therefore, if we consider again \eqref{topotornado}, we have: $$ \phi_{i+1}(\pi^S_{1,i+1}(u))\in \ker(L_{n-i-1})\ \text{because}\ \pi^S_{1,i+1}(u)\in \ker(D_i^*)$$ and $$\phi_{i+1}(\pi_{2,i+1}^S(u))\in \mathcal{D}(L_{n-i-1})\cap \im(L^*_{n-i-1})\ \text{because}\ \pi_{2,i+1}^S(u)\in \im(D_i)\cap \mathcal{D}(D_i^*).$$ In this way  we can conclude that: $$\pi^P_{1,n-i-1}(\phi_{i+1}(u))=\phi_{i+1}(\pi^S_{1,i+1}(u))\ \text{and}\ \pi_{2,n-i-1}^P(\phi_{i+1}(u))=\phi_{i+1}(\pi_{2,i+1}^S(u))$$ because $$\pi_{1,n-i-1}^P(\phi_{i+1}(u))+\pi_{2,n-i-1}^P(\phi_{i+1}(u))=\phi_{i+1}(\pi_{1,i+1}^S(u)+\pi_{2,i+1}^S(u))=\phi_{i+1}(u)$$ $$\pi_{1,n-i-1}^P(\phi_{i+1}(u))-\phi_{i+1}(\pi_{1,i+1}^S(u))\in \ker(L_{n-i-1}),$$ $$\pi_{2,n-i-1}^P(\phi_{i+1}(u))-\phi_{i+1}(\pi_{2,i+1}^S(u))\in \mathcal{D}(L_{n-i-1})\cap \im(L^*_{n-i-1})$$ and $$\ker(L_{n-i-1})\cap \im(L^*_{n-i-1})=\{0\}.$$ Thus we proved \eqref{toporazzo} and \eqref{toposaetta} and this means that $$\phi_{i+1}(\mathcal{D}(S_i))=\mathcal{D}(P_{n-i-1}).$$\\Now, in order to complete the proof of  \eqref{topocentrifuga}, we have to show that $$S_i(v)=\pm C_i^{-1}(\phi_{n-i}\circ P_{n-i-1}\circ \phi_{i+1})(v)$$ for each $v\in \mathcal{D}(S_i).$  Let $v\in \mathcal{D}(S_i)$. Then, according to \eqref{topotroppo},  $$v=v_1+v_2$$ with $v_1\in \ker(D_i^*)$, $v_2\in \pi_{2,i+1}^S(\mathcal{D}(L_i^*)\cap \im(L_i))$ and we have $\phi_i(S_i(v))=\phi_i(D_i^*(v_2))$ $=\phi_i(D_i^*(\pi_{2,i+1}^S(w)))$ where  $w\in \mathcal{D}(L_i^*)\cap \im(L_i)$  is unique because, as proved in the proof of Theorem \ref{cheafa}, $\pi_{2,i+1}^S$ is injective. So we get:    $$\phi_i(D_i^*(\pi_{2,i+1}^S(w)))=C_i^{-1}L_{n-i-1}(\phi_{i+1}(\pi_{2,i+1}^S(w)))=C_i^{-1}L_{n-i-1}(\pi_{2,n-i-1}^P(\phi_{i+1}(w)))=(\text{because}$$ $$ \pi_{2,n-i-1}^P(\phi_{i+1}(w))\in \mathcal{D}(P_{n-i-1}))=C_i^{-1}P_{n-i-1}(\pi_{2,n-i-1}^P(\phi_{i+1}(w)))=$$ $$C_i^{-1}P_{n-i-1}(\phi_{i+1}(\pi_{2,i+1}^S(w)))=C_i^{-1}P_{n-i-1}(\phi_{i+1}(v_2))=C_i^{-1}P_{n-i-1}(\phi_{i+1}(v_1+v_2))=$$ $$=C_i^{-1}P_{n-i-1}(\phi_{i+1}(v)).$$ Thus we proved that $$\phi_{i}(S_i(v))=C_i^{-1}P_{n-i-1}(\phi_{i+1}(v))$$ for each $v\in \mathcal{D}(S_i)$ and  therefore we can conclude that: $$ S_i=C_i^{-1}\phi_{i}^{-1}\circ P_{n-i-1}\circ \phi_{i+1}=\pm C_i^{-1}\phi_{n-i}\circ P_{n-i-1}\circ \phi_{i+1}.$$ 
\end{proof}

With the next theorem we investigate the Fredholm property for the complex $(H_i, P_i)$.

\begin{teo}
\label{chebrina}
Let  $(H_{j},D_{j})\subseteq (H_{j},L_{j})$,  $j=0,...,n$, be a pair of  Hilbert complexes. Suppose that, for each $j$, $\im(D_{j})$ is closed in $H_{j+1}$.  Let $(H_{j},P_{j})$ be the Hilbert complex built in Theorem \ref{cheafa}. Suppose that  $(H_{j},P_{j})$ is a Fredholm complex. Then:
\begin{enumerate}
\item  Every  Hilbert complex $(H_{j},T_{j})$, such that $(H_{j},D_{j})\subseteq (H_{j},T_{j})\subseteq (H_{j},L_{j})$, is a Fredholm complex.
\item The quotient of the domain of $L_j$ with the domain of $D_j$, that is $$\mathcal{D}(L_j)/\mathcal{D}(D_j)$$ is a finite dimensional vector space for each $j=0,...,n.$
\end{enumerate}
\end{teo}

\begin{proof}
Consider a  Hilbert complex $(H_{j},T_{j})$ such that $(H_{j},D_{j})\subseteq (H_{j},T_{j})\subseteq (H_{j},L_{j})$. For each $j$ we have a natural and injective map:
\begin{equation}
\label{fretui}
\ker(T_j)/\im(T_{j-1})\longrightarrow \ker(L_j)/\im(T_{j-1})
\end{equation}
and a natural and surjective map:
\begin{equation}
\label{fretuil}
\ker(L_j)/\im(D_{j-1})\longrightarrow \ker(L_j)/\im(T_{j-1})
\end{equation}
But, by the assumptions, we know that $\ker(L_j)/\im(D_{j-1})$, that is $H^j(H_*,P_*)$, is finite dimensional. Therefore, combining \eqref{fretui} and \eqref{fretuil} together, we get that $H^j(H_*,T_*)$ is finite dimensional for each $j$ and this completes the proof of the first statement.\\
Now, for every $j$, consider the two following vector spaces: $W_j$ defined as $$W_j:=\mathcal{D}(L_j)/\im(D_{j-1})$$ and $V_j$ defined as $$V_j:=\mathcal{D}(D_j)/\im(D_{j-1}).$$ Then $L_j$ induces a well defined operator, that we call $\tilde{L_j}$, acting  from $W_j$ to $\ker(L_{j+1})$. Analogously  $D_j$ induces a well defined operator $\tilde{D_j}:V_j\rightarrow \ker(L_j)$. Finally let us label by $\tilde{i}_j:V_j\longrightarrow W_j$ the map induced by the natural inclusion $i_j:\mathcal{D}(D_j)\rightarrow \mathcal{D}(L_j)$.\\Now we recall that on $W_j$ there is a natural and standard structure of Banach space because it  is defined as the quotient of a Hilbert space,  that is $(\mathcal{D}(L_j), \langle\ ,\ \rangle_{\mathcal{G}})$, with a closed subspace, that is $\im(D_{j})$. Analogously  also on $V_j$  there is a natural and standard structure of Banach space because it  is defined as the quotient of  $(\mathcal{D}(D_j), \langle\ ,\ \rangle_{\mathcal{G}})$,  which is a Hilbert space, with  $\im(D_{j})$, which is a closed subspace of $(\mathcal{D}(D_j), \langle\ ,\ \rangle_{\mathcal{G}})$ as well. We remark that the standard norm on $W_j$ is given by: $$\|[u]\|_{W_j}:=\inf_{s\in \im(D_j)}\|u+s\|_{(\mathcal{D}(L_j),\ \langle\ ,\ \rangle_{\mathcal{G}})}$$ where $[u]\in W_j$ and analogously on $V_j$ we have: $$\|[v]\|_{V_j}:=\inf_{s\in \im(D_j)}\|v+s\|_{(\mathcal{D}(D_j),\ \langle\ ,\ \rangle_{\mathcal{G}})}$$
where $[v]\in V_j$. It is immediate to check that in this way  we have three continuous operators:
\begin{equation}
\label{bismut} 
\tilde{L}_j:W_j\longrightarrow \ker(L_{j+1}),\ \tilde{D}_j:V_j\longrightarrow \ker(L_{j+1}),\ \tilde{i}_j:V_j\longrightarrow W_j
\end{equation}
acting between Banach spaces such that:
\begin{equation}
\label{pizzapo}
\tilde{L}_j\circ \tilde{i}_j=\tilde{D}_j.
\end{equation}
But $\tilde{D}_j$ is a Fredholm operator because $\ker(\tilde{D}_j)\cong H^j(H_*,D_*)$ and $\ck(\tilde{D}_j)\cong H^{j+1}(H_*,P_*)$. Analogously  $\tilde{L}_j$ is a Fredholm operator because $\ker(\tilde{L}_j)\cong H^j(H_*,P_*)$ and $\ck(\tilde{L}_j)\cong H^{j+1}(H_*,L_*)$. Therefore, combining with \eqref{pizzapo}, we get that $\tilde{i}_j$ is Fredholm too. But, by the definition of $\tilde{i}_{j}$, we get immediately that $\tilde{i}_j$ is injective and therefore we have \begin{equation}
\label{azzos}
\ind(\tilde{i}_j)=-\dim(\ck(\tilde{i}_j)).
\end{equation}
Now, in order to complete the proof, we have to observe  that 
\begin{equation}
\label{bivacco}
\ck(\tilde{i}_j)\cong W_j/\tilde{i}_j(V_j)\cong (\mathcal{D}(L_j)/\im(D_{j-1}))/(\tilde{i}_j(\mathcal{D}(D_j)/\im(D_{j-1})))\cong \mathcal{D}(L_j)/\mathcal{D}(D_j).
\end{equation} 
Therefore we can conclude that $$\mathcal{D}(L_j)/\mathcal{D}(D_j)$$ is a finite dimensional vector space and this establishes the theorem.
\end{proof}

\begin{cor}
\label{lupolupi}
Under the  assumptions of Theorem \ref{chebrina} we have the following \textbf{cohomological formula}:
\begin{equation}
\label{amoorso}
\dim(\mathcal{D}(L_j)/\mathcal{D}(D_j))=
\end{equation}
$$\dim(H^j(H_*,P_*))-\dim(H^{j+1}(H_*,L_*))+\dim(H^{j+1}(H_*,P_*))-\dim(H^j(H_*,D_*)).$$
\end{cor}

\begin{proof}
By \eqref{pizzapo}, \eqref{azzos} and \eqref{bivacco} we have $\dim(\mathcal{D}(L_j)/\mathcal{D}(D_j))=-\ind(\tilde{i}_j)=\ind(\tilde{L}_j)-\ind(\tilde{D}_j)$. But 
\begin{equation}
\label{chiara} 
\ind(\tilde{L}_j)= \dim(\ker(\tilde{L}_j))-\dim(\ck(\tilde{L}_j))=\dim(H^j(H_*,P_*))-\dim(H^{j+1}(H_*,L_*)).
\end{equation}
Analogously
\begin{equation}
\label{marcheggiani} 
\ind(\tilde{D}_j)= \dim(\ker(\tilde{D}_j))-\dim(\ck(\tilde{D}_j))=\dim(H^j(H_*,D_*))-\dim(H^{j+1}(H_*,P_*)).
\end{equation}
Therefore, combining \eqref{chiara} and \eqref{marcheggiani}, we get \eqref{amoorso} and this completes the proof.
\end{proof}

\section{Poincar\'e duality for $L^2$-cohomology}

Now we recall how Hilbert complexes appear naturally in the context of riemannian geometry.
Let $(M,g)$ be an open, oriented and possibly incomplete riemannian manifold.  Consider the de Rham complex $(\Omega_c^{*}(M),d_{*})$ where each form $\omega\in \Omega_c^{i}(M)$ is a $i-$form with compact support. Using the riemannian metric $g$ and the associated volume form $d\vol_g$ we can construct for each $i$ the Hilbert space $L^2\Omega^i(M,g).$ To turn the previous  complex into a Hilbert complex we must specify a closed extension of $d_i$. With the two following definitions we will recall the two canonical closed extensions of $d_i$.

\begin{defi} 
\label{maxi}
The maximal extension $d_{max}$; this is the operator acting on the domain:
\begin{equation} 
\mathcal{D}(d_{max,i})=\{\omega\in L^{2}\Omega^{i}(M,g): \exists\ \eta\in L^{2}\Omega^{i+1}(M,g)
\end{equation}

$$ s.t.\ \langle \omega,\delta_{i}\zeta\rangle_{L^{2}\Omega^i(M,g)}=\langle\eta,\zeta\rangle_{L^{2}\Omega^{i+1}(M,g)}\ \forall\ \zeta\in \Omega_c^{i+1}(M) \}.$$

In this case $d_{max,i}\omega=\eta.$ In  other words $\mathcal{D}(d_{max,i})$ is the largest set of forms $\omega\in L^{2}\Omega^{i}(M,g)$ such that $d_{i}\omega$, computed distributionally, is also in $L^{2}\Omega^{i+1}(M,g).$
\end{defi}

\begin{defi}
\label{mini}
 The minimal extension $d_{min,i}$; this is given by the graph closure of $d_{i}$ on $\Omega_c^{i}(M)$ with respect to the norm of $L^{2}\Omega^{i}(M,g)$, that is,
\begin{equation} \mathcal{D}(d_{min,i})=\{\omega\in L^{2}\Omega^{i}(M,g): \exists\ \{\omega_{j}\}_{j\in J}\subset \Omega_c^{i}(M,g),\ \omega_{j}\rightarrow \omega ,\       d_{i}\omega_{j}\rightarrow \eta\in L^{2}\Omega^{i+1}(M,g)\} 
\end{equation}
and in this case $d_{min,i}\omega=\eta.$
\end{defi}

Obviously $\mathcal{D}(d_{min,i})\subseteq \mathcal{D}(d_{max,i})$. Furthermore, from these definitions, we have immediately that $$d_{min,i}(\mathcal{D}(d_{min,i}))\subseteq \mathcal{D}(d_{min,i+1}),\ d_{min,i+1}\circ d_{min,i}=0$$ and that $$d_{max,i}(\mathcal{D}(d_{max,i}))\subseteq \mathcal{D}(d_{max,i+1}),\ d_{max,i+1}\circ d_{max,i}=0.$$ \\Therefore $(L^{2}\Omega^{*}(M,g),d_{max/min,*})$ are both Hilbert complexes and their cohomology groups are denoted by $H_{2,max/min}^{*}(M,g)$.\\Consider now the formal adjoint of $d_k$, $\delta_{k}:\Omega^{k+1}_c(M)\rightarrow \Omega^{k}_c(M)$. In completely analogy to the previous definition $\delta_{max,k}:L^2\Omega^{k+1}\rightarrow L^2\Omega^k(M,g)$ is defined as the distributional extension of $\delta_{k}$ while $\delta_{min,k}:L^2\Omega^{k+1}\rightarrow L^2\Omega^k(M,g)$ is defined as the graph closure of $\delta_k:\Omega^{k+1}_c(M)\rightarrow \Omega^k_c(M)$. A straightforward but important fact is that the Hilbert complex adjoint of \\$(L^{2}\Omega^{*}(M,g),d_{max/min,*})$ is $(L^{2}\Omega^{*}(M,g),\delta_{min/max,*})$, that is
\begin{equation}
(d_{max,i})^{*}=\delta_{min,i},\     (d_{min,i})^{*}=\delta_{max,i}.
\end{equation}

Using  Proposition \ref{beibei}  we obtain two weak Kodaira decompositions:
\begin{equation}
L^{2}\Omega^{i}(M,g)=\mathcal{H}^{i}_{abs/rel}(M,g)\oplus \overline{\im(d_{max/min,i-1})}\oplus \overline{\im(\delta_{min/max,i})}
\label{kd}
\end{equation}
with summands mutually orthogonal in each case. The first summand on the right, called the absolute or relative  Hodge cohomology, respectively, is defined as the orthogonal complement of the other two summands. Since $(\im(d_{max,i-1}))^{\bot}=\ker(\delta_{min,i-1})$ and $(\im(d_{min,i-1}))^{\bot}=\ker(\delta_{max,i-1})$, we see that 
\begin{equation}
\mathcal{H}^{i}_{abs/rel}=\ker(d_{max/min,i})\cap \ker(\delta_{min/max,i-1}).
\end{equation}

Now consider the following operators:
\begin{equation}
\Delta_{abs,i}=\delta_{min,i}d_{max,i}+d_{max,i-1}\delta_{min,i-1},\     \Delta_{rel,i}=\delta_{max,i}d_{min,i}+d_{min,i-1}\delta_{max,i-1}
\end{equation}
These are selfadjoint and satisfy:
\begin{equation}
\mathcal{H}^{i}_{abs}(M,g)=\ker(\Delta_{abs,i}),\ \mathcal{H}^{i}_{rel}(M,g)=\ker(\Delta_{rel,i})
\end{equation} and
\begin{equation}
\overline{\im(\Delta_{abs,i})}=\overline{\im(d_{max,i-1})}\oplus \overline{\im(\delta_{min,i})},\ \overline{\im(\Delta_{rel,i})}=\overline{ \im(d_{min,i-1})}\oplus \overline{\im(\delta_{max,i})}.
\end{equation}
Furthermore  if $H^{i}_{2,max/min}(M,g)$ is finite dimensional then the range of $d_{max/min,i-1}$ is closed and $\mathcal{H}^{i}_{abs/rel}(M,g)\cong H_{2,max/min}^{i}(M,g)$. On $L^{2}\Omega^{i}(M,g)$ we have also a third weak Kodaira decomposition:
\begin{equation}
L^{2}\Omega^{i}(M,g)=\mathcal{H}^{i}_{max}(M,g)\oplus \overline{\im(d_{min,i-1})}\oplus \overline{\im(\delta_{min,i})}
\label{oiko}
\end{equation}
where $\mathcal{H}^{i}_{max}(M,g)$ satisfies $\mathcal{H}^{i}_{max}(M,g)=\ker(d_{max,i})\cap \ker(\delta_{max,i-1})$. It is called the $i-th$ $maximal$ $Hodge$ $cohomology$ $group$.\\ Finally consider again the complex  $(\Omega_{c}^{*}(M),d_{*})$. We will call a \textbf{ closed extension} of  $(\Omega_{c}^{*}(M),d_{*})$ any Hilbert complex $(L^{2}\Omega^{i}(M,g), D_{i})$ where $D_{i}:L^{2}\Omega^{i}(M,g)\rightarrow L^{2}\Omega^{i+1}(M,g)$ is a  closed operator which extends $d_{i}:\Omega^{i}_{c}(M,g)\rightarrow \Omega^{i+1}_{c}(M,g)$ and such that the action of $D_{i}$ on $\mathcal{D}(D_{i})$, its domain, coincides with the action of $d_{i}$ on $\mathcal{D}(D_i)$ in the distributional sense. Obviously for every closed extension of  $(\Omega_{c}^{*}(M),d_{*})$ we have $(L^{2}\Omega^{*}(M,g),d_{min,*})\subseteq (L^{2}\Omega^{*}(M,g),D_{i})\subseteq (L^{2}\Omega^{*}(M,g),d_{max,*}). $ We will label with   $H^{i}_{2,D_{*}}(M,g)$ and $\overline{H}^{i}_{2,D_{*}}(M,g)$ respectively  the cohomology groups and the reduced cohomology group of $(L^{2}\Omega^{i}(M,g), D_{i})$ and with $\mathcal{H}^{i}_{D_{*}}(M,g)$ its Hodge cohomology groups.

 Now we are in the position to prove the following results:

\begin{prop}
\label{mario}
Let $(M,g)$ be an open, oriented and incomplete riemannian manifold of dimension $m$. Then the complexes $$(L^{2}\Omega^{*}(M,g),d_{max,*})\ and\  (L^{2}\Omega^{*}(M,g),d_{min,*})$$ are a pair of complementary Hilbert complexes. 
\\Moreover, for every $i=0,...,m$, we have the following isomorphism:   $$\ker(d_{max,i})/\overline{\im(d_{min,{i-1}})}\cong \ker(d_{max,m-i})/\overline{\im(d_{min,{m-i-1}})}.$$ 
\end{prop}

\begin{proof}
See \cite{FBE} Theorem 11 for the proof of the  first part of the theorem. The second part follows from Theorem \ref{berlino}.
\end{proof}

\begin{teo}
\label{polcas}
Let $(M,g)$ be an open, oriented and incomplete riemannian manifold of dimension $m$. Suppose that, for each $i=0,...,m,$ $\im(d_{min,i})$ is closed in $L^2\Omega^{i+1}(M,g)$. Then  there exists a Hilbert complex $(L^2\Omega^i(M,g)),d_{\mathfrak{M},i})$ such that,  for each $i=0,...m,$ the following properties are satisfied: 
\begin{itemize}
\item $\mathcal{D}(d_{min,i})\subseteq \mathcal{D}(d_{\mathfrak{M},i})\subseteq \mathcal{D}(d_{max,i}),$ that is $d_{max,i}$ is an extension of $d_{\mathfrak{M},i}$ which is an extension of $d_{min,i}$.  
\item $\im(d_{\mathfrak{M},i})$ is closed in $L^2\Omega^{i+1}(M,g)$.    
\item  If we denote $H^i_{2,\mathfrak{M}}(M,g)$  the cohomology of the Hilbert complex $(L^2\Omega^i(M,g),d_{\mathfrak{M},i})$ then we have: $$H^i_{2,\mathfrak{M}}(M,g)=\ker(d_{max,i})/\im(d_{min,i})$$ and $$H^i_{2,\mathfrak{M}}(M,g)\cong H^{m-i}_{2,\mathfrak{M}}(M,g).$$
\end{itemize}
\end{teo}

\begin{proof}
The proof  is an application of  Theorem \ref{cheafa} combined with  Proposition \ref{mario}.
\end{proof}

From Theorem \ref{dualcomplex} we get the following result:

\begin{teo}
\label{dualcomplexform}
Under  the  hypotheses of Theorem \ref{polcas}.  Consider the Hilbert complexes: 
\begin{equation}
\label{topoinfrarosso}
0\leftarrow L^2(M,g)\stackrel{\delta_{max,0}}{\leftarrow}L^2\Omega^1(M,g)\stackrel{\delta_{max,1}}{\leftarrow}L^2\Omega^{2}(M,g)\stackrel{\delta_{max,2}}{\leftarrow}...\stackrel{\delta_{max,n-1}}{\leftarrow}L^2\Omega^{n}(M,g)\leftarrow 0,
\end{equation}
and 
\begin{equation}
\label{topodellozio}
0\leftarrow L^2(M,g)\stackrel{\delta_{min,0}}{\leftarrow}L^2\Omega^1(M,g)\stackrel{\delta_{min,1}}{\leftarrow}L^2\Omega^{2}(M,g)\stackrel{\delta_{min,2}}{\leftarrow}...\stackrel{\delta_{min,n-1}}{\leftarrow}L^2\Omega^{n}(M,g)\leftarrow 0
\end{equation}
Let 
\begin{equation}
\label{topoultravioletto}
0\leftarrow L^2(M,g)\stackrel{\delta_{\mathfrak{M},0}}{\leftarrow}L^2\Omega^1(M,g)\stackrel{\delta_{\mathfrak{M},1}}{\leftarrow}L^2\Omega^{2}(M,g)\stackrel{\delta_{\mathfrak{M},2}}{\leftarrow}...\stackrel{\delta_{\mathfrak{M},n-1}}{\leftarrow}L^2\Omega^{n}(M,g)\leftarrow 0
\end{equation}
be the intermediate complex, which extends \eqref{topodellozio} and which is extended by \eqref{topoinfrarosso}, built according to Theorem \ref{cheafa}. Then, for each $i=0,...,m$,  we have:
\begin{equation}
\label{topopressato}
d_{\mathfrak{M},i}^*=\delta_{\mathfrak{M},i}=\pm *d_{\mathfrak{M},i}*.
\end{equation}
\end{teo}

\begin{proof}
It is an application of Theorem \ref{dualcomplex}. 
\end{proof}

Applying Theorem \ref{chebrina}  we get the following result:
\begin{teo}
\label{nuvoloso}
Let $(M,g)$ be an open, oriented and incomplete riemannian manifold of dimension $m$. Suppose that, for each $i=0,...,m$, $\im(d_{min,i})$ is closed in $L^2\Omega^{i+1}(M,g)$. Let $(L^2\Omega^i(M,g),d_{\mathfrak{M},i})$ be the Hilbert complex built in Theorem \ref{polcas}. Assume that $(L^2\Omega^i(M,g),d_{\mathfrak{M},i})$ is a Fredholm complex. Then:
\begin{enumerate}
\item  Every  closed extension $(L^2\Omega^i(M,g),D_i)$ of $(\Omega_c^i(M),d_i)$ is a Fredholm complex.
\item For every $i=0,...,\dim(M)$  the quotient of the domain of $d_{max,i}$ with the domain of $d_{min,i}$, that is $$\mathcal{D}(d_{max,i})/\mathcal{D}(d_{min,i})$$ is a finite dimensional vector space.
\end{enumerate}
\end{teo}

\begin{proof}
It is an application of Theorem \ref{chebrina}.
\end{proof}

Now, before stating the next result, we introduce some notations: 
\begin{defi}
\label{mnnmv}
 Let $(M,g)$ be an open, oriented and incomplete riemannian manifold of dimension $m$. Then, in analogy to the closed case, we  label with $b_{2,\mathfrak{M},i}(M,g), b_{2,M,i}(M,g)$ and  $b_{2,m,i}(M,g)$  respectively the dimension of $H^i_{2,\mathfrak{M}}(M,g), H^i_{2,max}(M,g)$ and $H^i_{2,min}(M,g)$ when they are finite dimensional. Moreover we define: 
\begin{equation}
\label{coccodrillo}
\chi_{2,M}(M,g):=\sum_{i=0}^m(-1)^ib_{2,M,i}(M,g),\ \chi_{2,\mathfrak{M}}(M,g):=\sum_{i=0}^m(-1)^ib_{2,\mathfrak{M},i}(M,g)
\end{equation}
and 
\begin{equation}
\label{orsolupo}
 \chi_{2,m}(M,g):=\sum_{i=0}^m(-1)^ib_{2,m,i}(M,g).
\end{equation}
\end{defi}

\begin{cor}
\label{lupolupe}
Under the  assumptions of Theorem \ref{nuvoloso} we have the  \textbf{cohomological formula}
\begin{equation}
\label{amoorsi}
\dim(\mathcal{D}(d_{max,i})/\mathcal{D}(d_{min,i}))=b_{2,\mathfrak{M},i}(M,g)-b_{2,M,i+1}(M,g)+b_{2,\mathfrak{M},i+1}(M,g)-b_{2,m,i}(M,g)
\end{equation}
 for each $i=0,...,m$.
\end{cor}

\begin{proof}
The corollary is an immediate application of Corollary \ref{lupolupi}.
\end{proof}

We conclude this section with the following result.  Let $(M,g)$ be an open, oriented and incomplete riemannian manifold of dimension $m$ such that $\im(d_{min,i})$ is closed for each $i=0,...,m$.  Assume that $(L^2\Omega^i(M,g),d_{\mathfrak{M},i})$ is a Fredholm complex. Then, according to Theorem \ref{nuvoloso}, we can define the following number associated to $(M,g)$:
\begin{equation}
\label{boiachi}
\psi_{L^2}(M,g):=\sum_{i=0}^m(-1)^i\dim(\mathcal{D}(d_{max,i})/\mathcal{D}(d_{min,i})).
\end{equation}

$\psi_{L^2}(M,g)$ satisfies the following properties:
\begin{teo}
\label{zarro}
Under the hypotheses of Theorem \ref{nuvoloso} the following formula holds:
\begin{equation}
\label{merisi}
\psi_{L^2}(M,g)=\chi_{2,M}(M,g)-\chi_{2,m}(M,g)=\left\{
\begin{array}{ll}
0\ &\ \dim(M)\ is\ even\\
2\chi_{2,M}(M,g)\ &\ \dim(M)\ is\ odd
\end{array}
\right.
\end{equation}
\end{teo}

\begin{proof}
By \eqref{amoorsi} and \eqref{boiachi} we have: $$\psi_{L^2}(M,g)=\sum_{i=0}^m(-1)^i(b_{2,\mathfrak{M},i}(M,g)-b_{2,M,i+1}(M,g)+b_{2,\mathfrak{M},i+1}(M,g)-b_{2,m,i}(M,g))=$$ $$=\sum_{i=0}^m(-1)^ib_{2,\mathfrak{M},i}(M,g)-\sum_{i=0}^m(-1)^ib_{2,M,i+1}(M,g)+\sum_{i=0}^m(-1)^ib_{2,\mathfrak{M},i+1}(M,g)-\sum_{i=0}^m(-1)^ib_{2,m,i}(M,g))=$$ $$=\chi_{2,\mathfrak{M}}(M,g)+\sum_{i=0}^m(-1)^{i+1}b_{2,M,i+1}(M,g)-\sum_{i=0}^m(-1)^{i+1}b_{2,\mathfrak{M},i+1}(M,g)-\chi_{2,m}(M,g)=$$ $$=\chi_{2,\mathfrak{M}}(M,g)+\sum_{i=1}^m(-1)^{i}b_{2,M,i}(M,g)-\sum_{i=1}^m(-1)^{i}b_{2,\mathfrak{M},i}(M,g)-\chi_{2,m}(M,g)=$$ $$=\chi_{2,\mathfrak{M}}(M,g)-1+1+\sum_{i=1}^m(-1)^{i}b_{2,M,i}(M,g)+1-1-\sum_{i=1}^m(-1)^{i}b_{2,\mathfrak{M},i}(M,g)-\chi_{2,m}(M,g)=$$ $$=\chi_{2,\mathfrak{M}}(M,g)-1+\chi_{2,M}(M,g)+1-\chi_{2,\mathfrak{M}}(M,g)-\chi_{2,m}(M,g)=$$ $$=\chi_{2,M}(M,g)-\chi_{2,m}(M,g).$$ This proves the first equality of \eqref{merisi}. By  Prop. \ref{mario}, we know that $(L^2\Omega^i(M,g),d_{max,i})$ and $(L^2\Omega^i(M,g),d_{min,i})$ are complementary Hilbert complexes. Moreover, by the assumptions, we know that these are both Fredholm. Therefore, applying Prop. \ref{errew}, we get $H^i_{2,max}(M,g)\cong H^{m-i}_{2,min}(M,g)$ for each $i=0,...,\dim(M)$. This implies that $\chi_{2,M}(M,g)=\chi_{2,m}(M,g)$ when $\dim(M)$ is even and that  $\chi_{2,M}(M,g)=-\chi_{2,m}(M,g)$ when $\dim(M)$ is odd. Thus we have proved the second equality of \eqref{merisi}.  
\end{proof}

\begin{cor}
\label{sinota}
Under the assumptions of Theorem \ref{nuvoloso} we have the  equality
\begin{equation}
\label{tortellino}
\psi_{L^2}(M,g)=\ind(d_{max}+\delta_{min})-\ind(d_{min}+\delta_{max})=\left\{
\begin{array}{ll}
0\ &\ \dim(M)\ is\ even\\
2\ind(d_{max}+\delta_{min})\ &\ \dim(M)\ is\ odd
\end{array}
\right.
\end{equation}
\end{cor}

\begin{proof}
We recall that $d_{max}+\delta_{min}$ is the operator associated to the complex $(L^2\Omega^i(M,g),d_{max,i})$ according to \eqref{aster}. Analogously, again according to \eqref{aster},  $d_{min}+\delta_{max}$ is the operator associated to the complex $(L^2\Omega^i(M,g),d_{min,i})$. Therefore they are Fredholm operators on their domains endowed with the graph norm and it is easy to see that $\ind(d_{max}+\delta_{min})=\chi_{2,M}(M,g)$ and that $\ind(d_{min}+\delta_{max})=\chi_{2,m}(M,g).$ Now, using \eqref{merisi},  we get the first equality of \eqref{tortellino}.\\ As in the proof of Theorem \ref{zarro}, we have  $\ind(d_{max}+\delta_{min})=\ind(d_{min}+\delta_{max})$ when $\dim(M)$ is even and $\ind(d_{max}+\delta_{min})=-\ind(d_{min}+\delta_{max})$ when $\dim(M)$ is odd. Therefore this implies immediately  the second equality of \eqref{tortellino} and so the proof is completed.
\end{proof}

\begin{rem}
A priori there is no reason to expect that $\dim(\mathcal{D}(d_{max,i})/\mathcal{D}(d_{min,i}))$ and $\psi_{L^2}(M,g)$ admit a description only in terms of $L^2$-cohomology groups. Therefore the identities \eqref{amoorsi} and \eqref{merisi} are remarkable. 
\end{rem}

\section{Other results concerning $(L^2\Omega^i(M,g),d_{\mathfrak{M},i})$}
In this section we collect other results concerning the Hilbert complex $(L^2\Omega^i(M,g),d_{\mathfrak{M},i})$. We start  with the following  \textbf{Hodge theorem} for the $L^2$-cohomology groups  $H^i_{2,\mathfrak{M}}(M,g)$.

\begin{teo}
\label{giove}
Under the  assumptions of Theorem \ref{polcas}; Let $\Delta_{i}:\Omega^i_{c}(M)\rightarrow \Omega_{c}^i(M)$ be the Laplacian acting on the space of smooth compactly supported $i-$forms. Then there exists a self-adjoint extension  $\Delta_{\mathfrak{M},i}:L^2\Omega^i(M,g)\rightarrow L^2\Omega^i(M,g)$ with closed range such that
\begin{equation}
\label{topoarrabbiato}
\ker(\Delta_{\mathfrak{M},i})\cong H^i_{2,\mathfrak{M}}(M,g)\cong \ker(d_{max,i})/\im(d_{min,i}).
\end{equation}
\end{teo}

\begin{proof}
Consider the Hilbert complex $(L^2\Omega^i(M,g),d_{\mathfrak{M},i})$. For each $i=0,...,m$ define 
\begin{equation}
\label{rosmarino}
\Delta_{\mathfrak{M},i}:=d_{\mathfrak{M},i}^*\circ d_{\mathfrak{M},i}+d_{\mathfrak{M},i-1}\circ d_{\mathfrak{M},i-1}^*
\end{equation}
with domain given by 
\begin{equation}
\label{sodio}
\mathcal{D}(\Delta_{\mathfrak{M},i})=\{\omega\in \mathcal{D}(d_{\mathfrak{M},i})\cap \mathcal{D}(d_{\mathfrak{M},i-1}^*):\ d_{\mathfrak{M},i}(\omega)\in \mathcal{D}(d_{\mathfrak{M},i}^*)\ \text{and}\ d_{\mathfrak{M},i-1}^*(\omega)\in \mathcal{D}(d_{\mathfrak{M},i-1})\}.
\end{equation}
In other words, for each $i=0,...,m$, $\Delta_{\mathfrak{M},i}$ is the $i-th$ Laplacian associated to the Hilbert complex $(L^2\Omega^i(M,g),d_{\mathfrak{M},i})$. So, as recalled in the first section,  we have that  \eqref{rosmarino}  is a  self-adjoint operator. Moreover, by Theorem \ref{polcas}, we know that $d_{\mathfrak{M},i}$ has closed range for each $i$.  This implies  that also $d_{\mathfrak{M},i}^*$ has closed range for each $i$. This means   that for the Hilbert complex $(L^2\Omega^i(M,g),d_{\mathfrak{M},i})$ the $L^2$-cohomology and the reduced $L^2$-cohomology are  the same and so we can apply \eqref{pppp} to get the first isomorphism of \eqref{topoarrabbiato}. The second one follows from  Theorem \ref{polcas}. Moreover, by the assumptions, we get that  $\im(\Delta_{\mathfrak{M},i})=\im(d_{\mathfrak{M},i-1})\oplus \im(d_{\mathfrak{M},i}^*)$.  Indeed  we have  $\im(\Delta_{\mathfrak{M},i})\subset \im(d_{\mathfrak{M},i-1})\oplus \im(d_{\mathfrak{M},i}^*)$ for all $i=0,...,m$.  Now let $\omega\in \im(d_{\mathfrak{M},i-1})\oplus \im(d_{\mathfrak{M},i}^*)$. Applying repeatedly the decomposition  in  Prop. \ref{beibei} and keeping in mind that $d_{\mathfrak{M},i}$ and  $d_{\mathfrak{M},i}^*$ have closed range in all degree, we get that $$\omega=d_{\mathfrak{M},i-1}(d_{\mathfrak{M},i-1}^*(d_{\mathfrak{M},i-1}(\eta_1)))+d_{\mathfrak{M},i}^*(d_{\mathfrak{M},i}(d_{\mathfrak{M},i}^*(\eta_2)))$$ for some $\eta_1\in \mathcal{D}(d_{\mathfrak{M},i-1})$ and $\eta_2\in \mathcal{D}(d_{\mathfrak{M},i}^*).$ Also, by the construction of $\eta_1$ and $\eta_2$, we get that  $$d_{\mathfrak{M},i-1}(\eta_1)+d_{\mathfrak{M},i}^*(\eta_2)\in \mathcal{D}(\Delta_{\mathfrak{M},i})$$ and $$d_{\mathfrak{M},i-1}(d_{\mathfrak{M},i-1}^*(d_{\mathfrak{M},i-1}(\eta_1)))+d_{\mathfrak{M},i}^*(d_{\mathfrak{M},i}(d_{\mathfrak{M},i}^*(\eta_2)))=\Delta_{\mathfrak{M},i}(d_{\mathfrak{M},i-1}(\eta_1)+d_{\mathfrak{M},i}^*(\eta_2)).$$ Therefore we get $\im(\Delta_{\mathfrak{M},i})\supset \im(d_{\mathfrak{M},i-1})\oplus \im(d_{\mathfrak{M},i}^*)$ and in  this way we can conclude  that  $\Delta_{\mathfrak{M},i}$ is an operator with closed range. This completes the proof.
\end{proof}

According to \eqref{said} we have $\ker(\Delta_{\mathfrak{M},i})=\ker(d_{\mathfrak{M},i})\cap \ker(d^*_{\mathfrak{M},i})$. We will label these spaces as $\mathcal{H}^{i}_{\mathfrak{M}}(M,g)$. Moreover, by the construction of $d_{\mathfrak{M,i}}$, we have that $\ker(d_{\mathfrak{M},i})=\ker(d_{max,i})$ and that $\im(d_{\mathfrak{M},i})=\im(d_{min,i})$. In particular this implies that the orthogonal decomposition of $L^2\Omega^i(M,g)$ induced by $(L^2\Omega^i(M,g),d_{\mathfrak{M},i})$, that is $$L^2\Omega^i(M,g)=\mathcal{H}^i_{\mathfrak{M}}(M,g)\oplus\im(d_{\mathfrak{M},i})\oplus\im(d_{\mathfrak{M},i}^*)$$ coincides with the one described in \eqref{oiko}, that is $$L^2\Omega^i(M,g)=\mathcal{H}^i_{max}(M,g)\oplus\im(d_{min,i})\oplus\im(\delta_{min,i}).$$ In particular we have \begin{equation}
\label{porcini}
\mathcal{H}^i_{max}(M,g)=\ker(\Delta_{\mathfrak{M},i})=\ker(d_{\mathfrak{M},i})\cap \ker(d^*_{\mathfrak{M},i}).
\end{equation}
The next proposition show that $H^i_{2,\mathfrak{M}}(M,g)$ is the biggest $L^2$-cohomology group for $(M,g).$

\begin{prop}
\label{sufcond}
Let $(M,g)$ be an open, oriented and incomplete riemannian manifold which satisfies the assumptions of Theorem \ref{polcas}. Then we have the following properties:
\begin{enumerate}
\item Consider the natural inclusion of complexes $ (L^2\Omega^i(M,g),d_{min,i})\subset  (L^2\Omega^i(M,g),d_{\mathfrak{M},i})$. Then the  map induced between cohomology groups is injective for all $i=0,...,m$.
\item Let $(L^{2}\Omega^{i}(M,g), D_{i})$ be a closed extension of $(\Omega_{c}^{i}(M),d_{i})$. Then, for each $i=0,...,m,$ there exists a natural injective map $\overline{H}^i_{2,D_*}(M,g)\rightarrow H^i_{2,\mathfrak{M}}(M,g)$.
\end{enumerate}
Finally, if $(L^2\Omega^i(M,g),d_{\mathfrak{M},i})$ is a Fredholm complex, then for every closed extension\\ $(L^2\Omega^i(M,g),D_i)$, there is an injective map $H^i_{2,D_*}(M,g)\rightarrow H^i_{2,\mathfrak{M}}(M,g)$ for every $i=0,...,m$.
\end{prop}

\begin{proof}
The first property follows immediately by the fact that $$\overline{H}^i_{2,\mathfrak{M}}(M,g)=\ker(d_{max,i})/\im(d_{min,i-1}).$$For the second property, by Prop. \ref{fred},  we have $\overline{H}^i_{2,D_*}(M,g)\cong \mathcal{H}^i_{D_*}(M,g)$. So applying \eqref{said} we get $\overline{H}^i_{2,D_*}(M,g)\cong \ker(D_i)\cap \ker(D_{i-1}^*)$. Applying the same statements to the complex $(L^2\Omega^i(M,g),d_{\mathfrak{M},i})$ we get $H^i_{2,\mathfrak{M}}(M,g)$ $\cong \ker(d_{\mathfrak{M},i})\cap \ker(d_{\mathfrak{M},i-1}^*)\cong$ $\ker(d_{max,i})\cap \ker(\delta_{max,i-1})$ by \eqref{porcini}. Summarizing we have:
 \begin{equation}
\label{assiria}
\overline{H}^i_{2,D_*}(M,g)\cong \ker(D_i)\cap \ker(D_{i-1}^*)\subset  \ker(d_{max,i})\cap \ker(\delta_{max,i-1})\cong H^i_{2,\mathfrak{M}}(M,g)
\end{equation}
and this proves the second statement.
Finally if  $(L^2\Omega^i(M,g),d_{\mathfrak{M},i})$ is a Fredholm complex then, according to Theorem \ref{nuvoloso}, we know that every closed extension $(L^2\Omega^i(M,g),D_i)$ is a Fredholm complex. Thus \eqref{assiria} becomes: $$H^i_{2,D_*}(M,g)\cong \ker(D_i)\cap \ker(D_{i-1}^*)\subset  \ker(d_{max,i})\cap \ker(\delta_{max,i-1})\cong H^i_{2,\mathfrak{M}}(M,g)$$ and this completes the proof.
\end{proof}

Finally we conclude this section with the following proposition:
\begin{prop}
\label{Dmeedm}
Let $(M,g)$ be an  oriented and incomplete riemannian  manifold. The  following properties are equivalent:
\begin{enumerate}
\item $\mathcal{D}(d_{min,i})=\mathcal{D}(d_{max,i})$ for all $i=0,...,m$.
\item  $\im(d_{min,i})=\im(d_{max,i})$ for all $i=0,...,n$.
\end{enumerate}
Moreover if $(L^2\Omega^i(M,g),d_{max/min,i})$ is a Fredholm complex then we have the following list of equivalent properties:
\begin{enumerate}
\item $\mathcal{D}(d_{min,i})=\mathcal{D}(d_{max,i})$ for all $i=0,...,m$.
\item $\im(d_{min,i})=\im(d_{max,i})$ for all $i=0,...,m$.
\item $\ker(d_{min,i})=\ker(d_{max,i})$ for all $i=0,...,m$.
\item $H^i_{2,max}(M,g)\cong H^i_{2,\mathfrak{M}}(M,g)$ for all $i=0,...,m$.
\item $H^i_{2,min}(M,g)\cong H^i_{2,\mathfrak{M}}(M,g)$ for all $i=0,...,m$.
\end{enumerate}
\end{prop}

\begin{proof}
We start proving the equivalence of the first pair of statements. Clearly $1)$ implies $2)$. Assume now that $2)$ holds. Then we know that also $\overline{\im(d_{min,i})}=\overline{\im(d_{max,i})}$ for all $i=0,...,m$. Therefore we get $\ker(\delta_{min,i})=\ker(\delta_{max,i})$ and finally, using the Hodge star operator we get $\ker(d_{min,i})=\ker(d_{max,i})$ for all $i=0,...,m$. Now let $\eta\in \mathcal{D}(d_{max,i})$. Then there exists $\omega\in \mathcal{D}(d_{min,i})$ such that $d_{max,i}\eta=d_{min,i\omega}$. This means that $\eta-\omega\in \ker(d_{max,i})$ and therefore there exists $\psi\in \ker(d_{min,i})$ such that $\eta-\omega=\psi $. Summarizing we got $\eta=\omega+\psi\in \mathcal{D}(d_{min,i})$ and this concludes the proof of the first part.\\Now we prove the second part of the proposition. First of all we observe that using the Hodge star operator, it follows easily that $(L^2\Omega^i(M,g),d_{max,i})$ is Fredholm if and only if $(L^2\Omega^i(M,g),d_{min,i})$ is Fredholm. Now from the first part we know that  the first two assertions are equivalent and they imply  the remaining statements. Assume now that $3)$ holds. Then applying the Hodge star operator we know that also $\ker(\delta_{min,i})=\ker(\delta_{max,i})$  and therefore that $\im(d_{min,i})=\im(d_{max,i})$ for all $i=0,...,m$ because, by the fact that $(L^2\Omega^i(M,g),d_{max/min,i})$ is a Fredholm complexes, we have that $\im(d_{max/min,i})$ is closed. So we can apply the first part of the proposition to get the conclusion.\\Now assume that $4)$ holds.  Then $H^i_{2,\mathfrak{M}}(M,g)$ is finite dimensional. We already know that $H^i_{2,max}(M,g)\cong \ker(\delta_{min,i-1})\cap \ker(d_{max,i})\subset \ker(\delta_{max,i-1})\cap \ker(d_{max,i})\cong H^i_{2,\mathfrak{M}}(M,g)$. Combining with $4)$ we get $$ \ker(\delta_{min,i-1})\cap \ker(d_{max,i})= \ker(\delta_{max,i-1})\cap \ker(d_{max,i})$$ and therefore using the weak Kodaira decompositions \eqref{kd} and \eqref{oiko} we have:  $$\im(d_{max,i-1})\oplus \im(\delta_{min,i})= \im(d_{min,i-1})\oplus \im(\delta_{min,i}).$$ In this way we get: $\im(d_{min,i-1})=  \im(d_{max,i-1})$ 
for each $i$. So we are in position to apply the first part of the proposition and therefore we proved that $4)\Rightarrow 1)$. In the same way, with the obvious modifications, we can prove  that $5)\Rightarrow 1)$.
\end{proof}

\section{$L^2-$Euler characteristic  and $L^2-$signature}

Let $(M,g)$ be an  open, oriented and incomplete riemannian manifold such that $(L^2\Omega^i(M,g), d_{\mathfrak{M},i})$ is a Fredholm complex. Then, in Definition \ref{mnnmv}, we defined  the $L^2-$Euler characteristic of $(M,g)$ associated to $(L^2\Omega^i(M,g), d_{\mathfrak{M},i})$ as: 
\begin{equation}
\chi_{2,\mathfrak{M}}(M,g):=\sum_{i=0}^m(-1)^{i}b_{2,\mathfrak{M},i}(M,g)
\end{equation}
where $b_{2,i,\mathfrak{M}}(M,g):=\dim(H^i_{2,\mathfrak{M}}(M,g))$.
We have the following immediate corollary:
\begin{cor}
Let $(M,g)$ be an open,  oriented and incomplete manifold such that $(L^2\Omega^i(M,g), d_{\mathfrak{M},i})$ is a Fredholm complex. If $m$ is odd then: $$\chi_{2,\mathfrak{M}}(M,g)=0.$$ 
\end{cor}
\begin{proof}
It is an immediate consequence of the fact that $H^i_{2,\mathfrak{M}}(M,g)\cong H^{m-i}_{\mathfrak{M}}(M,g).$
\end{proof}
Now consider the operator $$d_{\mathfrak{M}}+d_{\mathfrak{M}}^*:L^2\Omega^*(M,g)\rightarrow L^2\Omega^*(M,g)$$ defined according to \eqref{aster}. Let us label $L^2\Omega^{ev}(M,g):=\bigoplus_{i=0}^m L^2\Omega^{2i}(M,g)$ and analogously $L^2\Omega^{odd}(M,g):=\bigoplus_{i=0}^m L^2\Omega^{2i+1}(M,g)$. Define 
\begin{equation}
\label{stomp}
(d_{\mathfrak{M}}+d_{\mathfrak{M}}^*)^{ev/odd}:L^2\Omega^{ev/odd}(M,g)\rightarrow L^2\Omega^{odd/ev}(M,g)
\end{equation}
 as the restriction of $d_{\mathfrak{M}}+d_{\mathfrak{M}}^*$ to $L^2\Omega^{ev/odd}(M,g)$ with domain given by $$\mathcal{D}((d_{\mathfrak{M}}+d_{\mathfrak{M}}^*)^{ev}):= \mathcal{D}(d_{\mathfrak{M}}+d_{\mathfrak{M}}^*)\cap L^2\Omega^{ev}(M,g)=\bigoplus_{i=0}^m(\mathcal{D}(d_{\mathfrak{M},2i-1}^*)\cap \mathcal{D}(d_{\mathfrak{M},2i}))$$ and analogously 
$$\mathcal{D}((d_{\mathfrak{M}}+d_{\mathfrak{M}}^*)^{odd}):= \mathcal{D}(d_{\mathfrak{M}}+d_{\mathfrak{M}}^*)\cap L^2\Omega^{odd}(M,g)=\bigoplus_{i=0}^m(\mathcal{D}(d_{\mathfrak{M},2i}^*)\cap \mathcal{D}(d_{\mathfrak{M},2i+1}))$$
Clearly $(d_{\mathfrak{M}}+d_{\mathfrak{M}}^*)^{odd}$ is the adjoint of $(d_{\mathfrak{M}}+d_{\mathfrak{M}}^*)^{ev}.$

We are ready for the next theorem.

\begin{teo}
\label{eulerissimo}
Let $(M,g)$ be an open, oriented and incomplete riemannian manifold of dimension $m$ such that $(L^2\Omega^i(M,g), d_{\mathfrak{M},i})$ is a Fredholm complex. Then $(d_{\mathfrak{M}}+d_{\mathfrak{M}}^*)^{ev}$ is a  Fredholm operator on its  domain endowed with the graph norm and we have:
\begin{equation}
\label{oradicena}
\ind((d_{\mathfrak{M}}+d_{\mathfrak{M}}^*)^{ev})=\chi_{2,\mathfrak{M}}(M,g)
\end{equation}
\end{teo}

\begin{proof}
By the assumptions we know that $(L^2\Omega^i(M,g), d_{\mathfrak{M},i})$ is a Fredholm complex. Therefore, using Prop. \ref{worref}, we can conclude that  $d_{\mathfrak{M}}+d_{\mathfrak{M}}^*:L^2\Omega^*(M,g)\rightarrow L^2\Omega^*(M,g)$ is  Fredholm operator on its domain endowed with the graph norm. Clearly we have  $$\ker((d_{\mathfrak{M}}+d_{\mathfrak{M}}^*)^{ev})= \ker(d_{\mathfrak{M}}+d_{\mathfrak{M}}^*)\cap L^2\Omega^{ev}(M,g)$$ and $$\im((d_{\mathfrak{M}}+d_{\mathfrak{M}}^*)^{ev})=\im(d_{\mathfrak{M}}+d_{\mathfrak{M}}^*)\cap L^2\Omega^{odd}(M,g).$$ We get  immediately that $\ker((d_{\mathfrak{M}}+d_{\mathfrak{M}}^*)^{ev})$ is finite dimensional and that $\im((d_{\mathfrak{M}}+d_{\mathfrak{M}}^*)^{ev})$ is closed with finite dimensional orthogonal complement. So we got that also $(d_{\mathfrak{M}}+d_{\mathfrak{M}}^*)^{ev}$ is  a Fredholm  operator on its domain endowed with the graph norm. This implies that $(d_{\mathfrak{M}}+d_{\mathfrak{M}}^*)^{odd}$ is   Fredholm too because it is the adjoint of $(d_{\mathfrak{M}}+d_{\mathfrak{M}}^*)^{ev}$.\\ 
Now using  \eqref{said}, \eqref{kudam}, \eqref{pppp} we get: 
\begin{equation}
\label{selfcontrol}
\ker((d_{\mathfrak{M}}+d_{\mathfrak{M}}^*)^{ev})=\ker((d_{\mathfrak{M}}+d_{\mathfrak{M}}^*)^{odd}\circ(d_{\mathfrak{M}}+d_{\mathfrak{M}}^*)^{ev})=\sum_{i=0}^m \ker(\Delta_{\mathfrak{M},2i})=\sum_{i=0}^mH^{2i}_{2,\mathfrak{M}}(M,g).
\end{equation} Analogously: 
\begin{equation}
\label{outcontrol}
(\im((d_{\mathfrak{M}}+d_{\mathfrak{M}}^*)^{ev}))^{\bot}=\ker((d_{\mathfrak{M}}+d_{\mathfrak{M}}^*)^{odd})=\sum_{i=0}^m \ker(\Delta_{\mathfrak{M},2i+1})=\sum_{i=0}^mH^{2i+1}_{2,\mathfrak{M}}(M,g).
\end{equation} Now  \eqref{oradicena}  follows immediately by \eqref{selfcontrol} and \eqref{outcontrol} and this  establishes the Theorem.
\end{proof}

In the rest of this section we will describe how to define a $L^2-$signature for $(M,g)$ using $H^i_{2,\mathfrak{M}}(M,g)$. To this aim, first of all,  let us label $\overline{H}^i_{2,\mathfrak{M}}(M,g)$ the vector spaces defined as $$\overline{H}^i_{2,\mathfrak{M}}(M,g):=\ker(d_{max,i})/\overline{\im(d_{min,i})}.$$
The first step is to show  that using the wedge product we can construct a well defined and non degenerate pairing between $\overline{H}^{i}_{2,\mathfrak{M}}(M,g)$ and $\overline{H}^{m-i}_{2,\mathfrak{M}}(M,g)$ where $m=\dim(M)$.\\ 
We define:
\begin{equation}
\overline{H}^{i}_{2,\mathfrak{M}}(M,g)\times \overline{H}^{m-i}_{2,\mathfrak{M}}(M,g)\longrightarrow \mathbb{R},\  ([\eta],[\omega])\mapsto \int_{M} \eta\wedge \omega
\label{cambo}
\end{equation}
where $\omega$ and $\eta$ are any representative of $[\eta]$ and $[\omega]$ respectively.

\begin{prop}
Let (M,g) be an open, oriented and incomplete riemannian manifold of dimension $m$.
Then \eqref{cambo} is a well defined and  non degenerate pairing.
\end{prop}

\begin{proof}
The first step is to show that \eqref{cambo} is well defined. Let $\eta',\ \omega'$ other two forms such that $[\eta]=[\eta']$ in $\overline{H}^{i}_{2,\mathfrak{M}}(M,g)$,   $[\omega]=[\omega']$ in $\overline{H}^{m-i}_{2,\mathfrak{M}}(M,g)$. 
 Then there exists $\alpha \in \overline{\im(d_{min,i-1})}$ and  $\beta \in \overline{\im(d_{min,m-i-1})}$ such that $\eta =\eta'+\alpha$ and $\omega=\omega'+\beta$. Therefore: $$\int_{M}\eta\wedge\omega=\int_{M}(\eta'+\alpha)\wedge(\omega'+\beta)=\int_{M}\eta'\wedge\omega'+\int_{M}\eta'\wedge\beta+\int_{M}\alpha\wedge\omega'+\int_{M}\alpha\wedge\beta$$
Now $$\int_{M}\eta'\wedge\beta=\pm\int_{M}\langle\eta',*\beta\rangle d\vol_{g}=\pm\langle\eta',*\beta\rangle_{L^{2}\Omega^i(M,g)}=0$$ because $*\beta\in \im(\delta_{min,i})$ and $\alpha\in \ker(d_{max,i})$. In the same way: $$\int_{M}\alpha\wedge\beta=\pm\int_{M}\langle\alpha,*\beta\rangle d\vol_{g}=\pm\langle\alpha,*\beta\rangle_{L^{2}\Omega^i(M,g)}=0$$ because $\alpha\in \im(d_{min,i-1})$ and $*\beta\in \im(\delta_{min,i})$. Finally 
$$\int_{M}\alpha\wedge\omega'=\pm\int_{M}\langle\alpha,*\omega'\rangle d\vol_{g}=\pm\langle\alpha,*\omega'\rangle_{L^{2}\Omega^i(M,g)}=0$$ because $\alpha\in \im(d_{min,i-1})$ and $\omega'\in \ker(\delta_{max,i-1})$.  So we can conclude that \eqref{cambo} is well defined. Now fix $[\eta] \in \overline{H}^{i}_{2,\mathfrak{M}}(M,g)$ and suppose that for each  $[\omega] \in \overline{H}^{m-i}_{2,\mathfrak{M}}(M,g)$ the pairing \eqref{cambo} vanishes. Then this means that for each $\omega\in \ker(d_{max,m-i})$ we have $\int_{M}\eta\wedge \omega=0$. We also know that $\int_{M}\eta\wedge \omega=\pm\langle\eta,*\omega\rangle_{L^{2}\Omega^i(M,g)}$ and that  $*(\ker(d_{max,m-i}))=\ker(\delta_{max,i-1})$ . So by the fact that $(
\ker(\delta_{max,i-1}))^{\bot}=\overline{\im(d_{min,i-1})}$ we obtain that $[\eta]=0$. In the same way if  $[\omega]\in \overline{H}^{m-i}_{2,\mathfrak{M}}(M,g)$ is such that  for each  $[\eta] \in \overline{H}^{i}_{2,\mathfrak{M}}(M,g)$ the pairing \eqref{cambo} vanishes then we know that  for each $\eta\in \ker(d_{max,i})$ we have $\int_{M}\eta\wedge \omega=0$. But we know that  $\int_{M}\eta\wedge \omega=\pm\langle\eta,*\omega\rangle_{L^{2}\Omega^i(M,g)}$. By the fact   that $(\ker(d_{max,i}))^{\bot}=\overline{\im(\delta_{min,i})}$ and that $*(\overline{\im(\delta_{min,i})})=(\overline{\im(d_{min,m-i-1})})$ we obtain that $[\omega]=0$.\\  So we can conclude that the pairing \eqref{cambo} is well defined and non degenerate and this establishes the proposition.
\end{proof}

We have the following immediate corollary:

\begin{cor}
Let $(M,g)$ be an open,  oriented and incomplete riemannian manifold of dimension $m=4n$. Then on $\overline{H}^{2n}_{2,\mathfrak{M}}(M,g)$ the pairing \eqref{cambo} is a symmetric bilinear form.
\end{cor}

 We can now state  the following definition:
\begin{defi}
\label{qwerq}
Let $(M,g)$ be an open,  oriented and incomplete riemannian manifold of dimension $m=4n$ such that, for  $i=2n$, $\overline{H}^{2n}_{2,\mathfrak{M}}(M,g)$ is finite dimensional. Then we define the $L^{2}-$signature of $(M,g)$ associated to $\overline{H}^{2n}_{2,\mathfrak{M}}(M,g)$ \footnote{In \cite{FBE} we introduced a different $L^2-$signature for $(M,g)$ using  another kind of $L^2$-cohomology. So when $(M,g)$ is incomplete we may have different kinds of $L^2-$signatures and therefore we have to specify the $L^2-$complex that we are using.}  and we label it $\sigma_{2,\mathfrak{M}}(M,g)$ as the signature of the pairing \eqref{cambo} applied on $\overline{H}^{2n}_{2,\mathfrak{M}}(M,g)$.
\label{nnbb}
\end{defi} 

Before  concluding  this section with the next theorem we need to introduce some notations. Let $(M,g)$ be an open, oriented and incomplete riemannian manifold of dimension $m=4l.$ Consider the complexified cotangent bundle $T^*_{\mathbb{C}}M\cong T^*M\otimes \mathbb{C}.$ Then the metric $g$ admits a natural extension as  a positive definite hermitian metric on $T^*M\otimes \mathbb{C}$ and therefore, in complete analogy to the real case, we can build $L^2\Omega_{\mathbb{C}}^*(M,g)\cong L^2\Omega^*(M,g)\otimes \mathbb{C}$, $d_{max/\mathfrak{M}/min,i}:L^2\Omega_{\mathbb{C}}^i(M,g)\longrightarrow L^2\Omega_{\mathbb{C}}^{i+1}(M,g)$, $(d+\delta)_{max/min}:L^2\Omega_{\mathbb{C}}^*(M,g)\longrightarrow L^2\Omega_{\mathbb{C}}^*(M,g)$ etc, etc. Consider now the endomorphism $\epsilon:\Lambda_{\mathbb{C}}^*(T^*M) \longrightarrow  \Lambda_{\mathbb{C}}^*(T^*M)$ defined by $\epsilon:=(\sqrt{-1})^{p(p-1)+2l}*$ on $\Lambda_{\mathbb{C}}^p(T^*M)$. This is the well known endomorphism of the classical signature theorem. In fact we have $\epsilon^2=Id$ and therefore we get  the well known $\mathbb{Z}_2$ graduation of the signature theorem  given by the eigenspaces of $\epsilon$ associated to eigenvalues $\{\pm1\}$: $\Lambda^*_{\mathbb{C}}(M)\cong (\Lambda^*_{\mathbb{C}}(M))^+\oplus (\Lambda^*_{\mathbb{C}}(M))^-$, $\Omega^*(M,\mathbb{C})\cong (\Omega^*(M,\mathbb{C}))^+\oplus (\Omega^*(M,\mathbb{C}))^-$. Clearly we can extend this $\mathbb{Z}_2$ graduation also in the $L^2$ setting and we get $L^2\Omega_{\mathbb{C}}^*(M,g)\cong (L^2\Omega_{\mathbb{C}}^*(M,g))^+\oplus (L^2\Omega_{\mathbb{C}}^*(M,g))^-$. Another well known property is that $d+\delta$ is odd with respect to $\epsilon$. So we can recall the definition of the \emph{signature operator} as the operator acting in the following way: $$d+\delta:(\Omega_c^*(M,\mathbb{C}))^+\longrightarrow (\Omega_c^*(M,\mathbb{C}))^-.$$ We  label it $D^{sign,+}$. Clearly $D^{sign,-}$, that is $d+\delta:(\Omega_c^*(M,\mathbb{C}))^-\longrightarrow (\Omega_c^*(M,\mathbb{C}))^+$, is the formal adjoint of $D^{sign,+}$. Finally we introduce: $$\Delta^+:= D^{sign,-}\circ D^{sign,+},\ \Delta^+:(\Omega_c^*(M,\mathbb{C}))^+\longrightarrow (\Omega_c^*(M,\mathbb{C}))^+$$ and  $$\Delta^-:= D^{sign,+}\circ D^{sign,-},\ \Delta^-:(\Omega_c^*(M,\mathbb{C}))^-\longrightarrow (\Omega_c^*(M,\mathbb{C}))^-.$$ Our goal now is to define a closed extension  of $D^{sign,+}$ which is a Fredholm  operator on its domain endowed with the graph norm and whose index equals $\sigma_{2,\mathfrak{M}}(M,g)$. In order to get this aim consider again the following operators:
\begin{equation}
\label{ilocchi}
d_{\mathfrak{M}}+d_{\mathfrak{M}}^*:L^2\Omega_{\mathbb{C}}^*(M,g)\longrightarrow L^2\Omega_{\mathbb{C}}^*(M,g)
\end{equation}
 and  $$\Delta_{\mathfrak{M}}:=(d_{\mathfrak{M}}+d_{\mathfrak{M}}^*)\circ (d_{\mathfrak{M}}+d_{\mathfrak{M}}^*).$$  From Theorems  \ref{dualcomplex} and \ref{dualcomplexform} we get that both $\mathcal{D}(d_{\mathfrak{M}}+d_{\mathfrak{M}}^*)$ and $\mathcal{D}(\Delta_{\mathfrak{M}})$ are invariant under the action of the Hodge star operator $*$. Therefore they are also  invariant under the action of  $\epsilon$. Moreover $d_{\mathfrak{M}}+d_{\mathfrak{M}}^*$ is odd with respect to $\epsilon$. In particular we get
\begin{equation}
\label{muller}
\mathcal{D}(d_{\mathfrak{M}}+d_{\mathfrak{M}}^*)=(\mathcal{D}(d_{\mathfrak{M}}+d_{\mathfrak{M}}^*))^+\oplus (\mathcal{D}(d_{\mathfrak{M}}+d_{\mathfrak{M}}^*))^-
\end{equation}
and analogously 
\begin{equation}
\label{mullerq}
\mathcal{D}(\Delta_{\mathfrak{M}})=(\mathcal{D}(\Delta_{\mathfrak{M}}))^+\oplus (\mathcal{D}(\Delta_{\mathfrak{M}}))^-
\end{equation}
Define now $(d_{\mathfrak{M}}+d_{\mathfrak{M}}^*)^{+/-}$ as the restriction of $d_{\mathfrak{M}}+d_{\mathfrak{M}}^*$ to $(\mathcal{D}(d_{\mathfrak{M}}+d_{\mathfrak{M}}^*))^{+/-}$ respectively. Therefore we have $$(d_{\mathfrak{M}}+d_{\mathfrak{M}}^*)^+:(\mathcal{D}(d_{\mathfrak{M}}+d_{\mathfrak{M}}^*))^+\longrightarrow  (L^2\Omega_{\mathbb{C}}^*(M,g))^-$$ and analogously $$(d_{\mathfrak{M}}+d_{\mathfrak{M}}^*)^-:(\mathcal{D}(d_{\mathfrak{M}}+d_{\mathfrak{M}}^*))^-\longrightarrow  (L^2\Omega_{\mathbb{C}}^*(M,g))^+.$$ Clearly
$(d_{\mathfrak{M}}+d_{\mathfrak{M}}^*)^{-}$ is the adjoint of  $(d_{\mathfrak{M}}+d_{\mathfrak{M}}^*)^{+}.$ Define $\Delta_{\mathfrak{M}}^{+}$ as $\Delta_{\mathfrak{M}}^{+}:=(d_{\mathfrak{M}}+d_{\mathfrak{M}}^*)^-\circ (d_{\mathfrak{M}}+d_{\mathfrak{M}}^*)^+$ and analogously $\Delta_{\mathfrak{M}}^{-}:=(d_{\mathfrak{M}}+d_{\mathfrak{M}}^*)^+\circ (d_{\mathfrak{M}}+d_{\mathfrak{M}}^*)^-$. Finally we are in position to prove the last theorem of this section:
\begin{teo}
\label{luxlux}
Let $(M,g)$ be an  open, oriented and incomplete riemannian manifold of dimension $4l$ such that $(L^2\Omega^i(M,g),d_{\mathfrak{M},i})$ is a Fredholm complex.  Then   $(d_{\mathfrak{M}}+d_{\mathfrak{M}}^*)^{+}$ is a Fredholm operator on its domain endowed with the graph norm and  we have: $$\sigma_{2,\mathfrak{M}}(M,g)=\ind((d_{\mathfrak{M}}+d_{\mathfrak{M}}^*)^+).$$
\end{teo}

\begin{proof} 
By the assumptions $d_{\mathfrak{M}}+d_{\mathfrak{M}}^*$  is a Fredholm operator on its domain endowed with the graph norm.  By the fact that $$\ker((d_{\mathfrak{M}}+d_{\mathfrak{M}}^*)^{+/-})=\ker(d_{\mathfrak{M}}+d_{\mathfrak{M}}^*)\cap (L^2\Omega^{*}_{\mathbb{C}}(M,g))^{+/-}$$ and that $$\im((d_{\mathfrak{M}}+d_{\mathfrak{M}}^*)^{+/-})=\im(d_{\mathfrak{M}}+d_{\mathfrak{M}}^*)\cap (L^2\Omega^{*}_{\mathbb{C}}(M,g))^{+/-}$$  we get that also $(d_{\mathfrak{M}}+d_{\mathfrak{M}}^*)^{+/-}$ are Fredholm operators on their respective domains endowed with the graph norm. This proves the first part of the proposition.\\Now, in order to prove the second part, we follows, with the necessary modifications,  the classic proof of the signature Theorem, see for example \cite{BGV}.
We start observing that $$\ind((d_{\mathfrak{M}}+d_{\mathfrak{M}}^*)^+)=\dim(\ker(\Delta_{\mathfrak{M}}^+))-\dim(\ker(\Delta_{\mathfrak{M}}^-)).$$ Moreover we have: $$\ker(\Delta_{\mathfrak{M}}^{+/-})=(\bigoplus_{k=0}^{2l-1}(\ker(\Delta_{\mathfrak{M}}^{+/-})\cap(L^2\Omega_{\mathbb{C}}^k(M,g)\oplus L^2\Omega_{\mathbb{C}}^{4l-k}(M,g))))\oplus (\ker(\Delta_{\mathfrak{M}}^{+/-})\cap L^2\Omega_{\mathbb{C}}^{2l}(M,g)).$$
Now if $\omega \in \ker(\Delta_{\mathfrak{M}}^{+})\cap(L^2\Omega_{\mathbb{C}}^k(M,g)\oplus L^2\Omega_{\mathbb{C}}^{4l-k}(M,g))$ with $k\leq 2l-1$ then $\omega=\eta+\epsilon(\eta)$ with $\eta\in \mathcal{H}^k_{\mathfrak{M}}(M,g).$ On the other hand if $\eta \in  \mathcal{H}^k_{\mathfrak{M}}(M,g)$ then $\eta+\epsilon(\eta)\in \ker(\Delta_{\mathfrak{M}}^{+})\cap(L^2\Omega_{\mathbb{C}}^k(M,g)\oplus L^2\Omega_{\mathbb{C}}^{4l-k}(M,g))$. Therefore we can conclude that  $$\ker(\Delta_{\mathfrak{M}}^{+})\cap(L^2\Omega_{\mathbb{C}}^k(M,g)\oplus L^2\Omega_{\mathbb{C}}^{4l-k}(M,g))=\{\eta+\epsilon(\eta),\ \eta\in \mathcal{H}^k_{\mathfrak{M}}(M,g)\}.$$ The same observations lead to the conclusion that  $$\ker(\Delta_{\mathfrak{M}}^{-})\cap(L^2\Omega_{\mathbb{C}}^k(M,g)\oplus L^2\Omega_{\mathbb{C}}^{4l-k}(M,g))=\{\eta-\epsilon(\eta),\ \eta\in \mathcal{H}^k_{\mathfrak{M}}(M,g)\}.$$ 
In this way we get  that $$\bigoplus_{k=0}^{2l-1}(\ker(\Delta_{\mathfrak{M}}^{+})\cap(L^2\Omega_{\mathbb{C}}^k(M,g)\oplus L^2\Omega_{\mathbb{C}}^{4l-k}(M,g)))$$ is isomorphic to $$\bigoplus_{k=0}^{2l-1}(\ker(\Delta_{\mathfrak{M}}^{-})\cap(L^2\Omega_{\mathbb{C}}^k(M,g)\oplus L^2\Omega_{\mathbb{C}}^{4l-k}(M,g))).$$ So we proved that $$\ind((d_{\mathfrak{M}}+d_{\mathfrak{M}}^*)^+)=$$ $$=\dim(\ker(\Delta_{\mathfrak{M}}^{+})\cap L^2\Omega_{\mathbb{C}}^{2l}(M,g))-\dim(\ker(\Delta_{\mathfrak{M}}^{-})\cap L^2\Omega_{\mathbb{C}}^{2l}(M,g)).$$ But $\ker(\Delta_{\mathfrak{M}})\cap L^2\Omega_{\mathbb{C}}^{2l}(M,g)=\mathcal{H}^{2l}_{\mathfrak{M}}(M,g)$ and this implies that $\ker(\Delta_{\mathfrak{M}}^{+/-})\cap L^2\Omega_{\mathbb{C}}^{2l}(M,g)=(\mathcal{H}^{2l}_{\mathfrak{M}}(M,g))^{+/-}.$ Now if $\eta\in (\mathcal{H}^{2l}_{\mathfrak{M}}(M,g))^{+}$ this means that $\eta \in \mathcal{H}^{2l}_{\mathfrak{M}}(M,g)$ and that $\epsilon(\eta)=\eta$ that is $*\eta=\eta.$ Analogously if $\eta\in (\mathcal{H}^{2l}_{\mathfrak{M}}(M,g))^{-}$ this means that $\eta \in \mathcal{H}^{2l}_{\mathfrak{M}}(M,g)$ and that $\epsilon(\eta)=-\eta$ that is $*\eta=-\eta.$ In conclusion   we proved that: $$\ind((d_{\mathfrak{M}}+d_{\mathfrak{M}}^*)^+)=\dim(\ker(\Delta_{\mathfrak{M}}^{+})\cap L^2\Omega_{\mathbb{C}}^{2l}(M,g))-\dim(\ker(\Delta_{\mathfrak{M}}^{-})\cap L^2\Omega_{\mathbb{C}}^{2l}(M,g))=$$ $$=\dim(\mathcal{H}^i_{\mathfrak{M}}(M,g))^{+}-\dim(\mathcal{H}^i_{\mathfrak{M}}(M,g))^{-}=\sigma_{2,\mathfrak{M}}(M,g)$$ and this completes the proof.
\end{proof}

\section{Some examples and applications}

It is  not difficult to find examples of  open, oriented and incomplete riemannian manifolds $(M,g)$ of dimension $m$ such that $\im(d_{min,i})$ is closed in $L^2\Omega^{i+1}(M,g)$ for all $i=0,...,m$. We can consider, for example,  a compact and oriented manifold with boundary endowed with a smooth metric up to the boundary  as in \cite{BL}, admissible riemannian pseudomanifold as in \cite{C} or in \cite{MN}, compact stratified pseudomanifold endowed with a \emph{quasi edge metric with weights} as in \cite{FB} or  the Weil-Peterson metric on the regular part of the moduli space of curves as in \cite{LS}.   In these examples  the maximal $L^2$-de Rham cohomology, $H^i_{2,max}(M,g)$, is finite dimensional for each $i=0,...,m$. As explained in the proof of Theorem \ref{zarro} this implies that $H^i_{2,min}(M,g)\cong H^{m-i}_{2,max}(M,g)$ and therefore $H^i_{2,min}(M,g)$ is finite dimensional as well.  Finally, as recalled in Prop. \ref{topoamotore}, we can conclude that  $\im(d_{min,i})$ is closed in $L^2\Omega^{i+1}$ for each $i=0,...,m$. Therefore, in all these cases, we can always build the complex $(L^2\Omega^i(M,g),d_{\mathfrak{M},i})$. What is much more complicated is to find examples of  open, oriented and incomplete riemannian manifolds $(M,g)$ such that $(L^2\Omega^i(M,g), d_{\mathfrak{M},i})$ is a Fredholm complex.  The first part of this last section is devoted to this task.\\ First of all we recall that two riemannian metrics $g$ and $h$ are said \emph{quasi isometric}  if there exists a positive real number $c$ such that $\frac{1}{c}h\leq g\leq ch$. It is easy to check that if $M$ is an oriented manifold of dimension $m$ and if $g$ and $h$ are two riemannian metrics over $M$ quasi-isometric then, for every $i=0,...,m$, $L^2\Omega^i(M,g)=L^2\Omega^i(M,h)$, $\mathcal{D}(d_{max,i})$, $\ker(d_{max,i})$, $\im(d_{max,i})$ (with respect to $g$) coincide respectively with $\mathcal{D}(d_{max,i})$, $\ker(d_{max,i})$, $\im(d_{max,i})$ (with respect to $h$) and analogously 
$\mathcal{D}(d_{min,i})$, $\ker(d_{min,i})$, $\im(d_{min,i})$ (with respect to $g$) coincide respectively with $\mathcal{D}(d_{min,i})$, $\ker(d_{min,i})$, $\im(d_{min,i})$ (with respect to $h$).\\
Now we describe the first example; we start with  the following definition from \cite{BrL}.\\Let $\overline{M}$ be a compact manifold with boundary $N:=\partial{\overline{M}}$. Let us label its interior with $M$.  Let $U\cong [0,1)\times N$ be a collar neighborhood for $N$. Let $g$ be a riemannian metric over $M$ such that $g$ restricted to $U$ is isometric to $h(x)(dx^2+x^2g_{N}(x))$ where $g_{N}(x)$ is a family  of metric on $N$ depending on $x$ which varies smoothly
in $(0,1)$ and continuously $[0,1)$ and $h\in C^{\infty}((0,1)\times N)$ satisfies:
$$\sup_{p\in N}|(x\partial_x)^jx^{-c}h(x,p)-1|=O(x^{\delta})\ \text{as}\ x\rightarrow 0,\ j=0,1$$ 
$$\sup_{p\in N}\|h(x,p)^{-1}d_Nh(x,p)\|_{T_p^*N,g_{N}(x)}=O(x^{\delta})\ \text{as}\ x\rightarrow 0$$ and 
$$\sup_{p\in N}(|(g^1-g^0)|_{(x,p)}+x|\omega^0-\omega^1|_{(x,p)})=O(x^{\delta})\ \text{as}\ x\rightarrow 0$$
for some $\delta>0$ and $c>-1$ and where $g^0:=dx^2+x^2g_N(0)$, $g_1=dx^2+x^2g_N(x)$ and $\omega^0, \omega^1$ are the connection forms of the Levi-Civita connection $\nabla^0, \nabla^1$ of $g^0$ and $g^1$ respectively.\\
 The metric $g$ is called a \textbf{conformally conic metric}.  As it is showed in \cite{BruL}, \cite{BrLe} and \cite{BPS} if we consider a complex projective curve $V\subset \mathbb{C}\mathbb{P}^n$  and $g$ is the riemannian metric induced by the Fubini-Study metric of $\mathbb{C}\mathbb{P}^n$ on the regular part of $V$ then $g$ is a conformally conic metric.\\According to \cite{BrL} if $(M,g)$ is a conformally conic riemannian manifold then {\em every closed extension of} $(\Omega^i_c(M), d_i)$ {\em is a Fredholm complex}. In particular, according to Prop. \ref{topoamotore}, $\im(d_{min,i})$ is closed for each $i$.  Therefore we have the following corollary:
\begin{cor}
\label{coroll}
Let $(M,g)$ be an oriented riemannian manifold where $M$ is the interior of a compact manifold with boundary and $g$ is a riemannian metric on $M$ quasi isometric to a  conformally conic metric. Then Theorems \ref{polcas}, \ref{nuvoloso}, \ref{zarro},  \ref{giove},  \ref{eulerissimo}, and \ref{luxlux} and their relative corollaries hold for $(M,g)$. In particular they hold when $M$ is the regular part of  a complex projective curve $V\subset \mathbb{C}\mathbb{P}^n$  and $g$ is any riemannian metric on $M$ quasi isometric to the metric induced by the Fubini-Study metric of $\mathbb{C}\mathbb{P}^n$. 
\end{cor}

Another example is the following: consider again a compact and oriented riemannian manifold with boundary $\overline{M}$. As above let us label with $N$ the boundary of $\overline{M}$, with $M$ its interior and finally with $U\cong [0,1)$ a collar  neighborhood of $N$. Let $g$ be a riemannian metric on $M$ such that, over $U$, it takes the form $dx^2+x^{2\beta}h$ where $\beta>1$ and $h$ is a riemannian metric on $N$. A riemannian metric like that is called \textbf{metric horn}. In \cite{LP} the authors prove that if we consider  the Gauss-Bonnet operator 
\begin{equation}
\label{hurleybona}
d+\delta:L^2\Omega^*(M,g)\rightarrow L^2\Omega^*(M,g)
\end{equation}
 with domain given by $\Omega^*_c(M)$ then every closed extension of  \eqref{hurleybona} is a Fredholm operator on its domain endowed with the graph norm. This, according to  Lemma 2.3 of \cite{BL} and to Prop. \ref{worref}, implies that {\em every closed extension of} $(\Omega^i_c(M,g),g)$ {\em is a Fredholm complex}. In particular, according to Prop. \ref{topoamotore}, $\im(d_{min,i})$ is closed for each $i$. Therefore we have:
\begin{cor}
\label{corolll}
Let $(M,g)$ be an  oriented riemannian manifold where $M$ is the interior of a compact manifold with boundary and let $g$ be a riemannian metric on $M$  quasi-isometric to a  metric horn. Then Theorems \ref{polcas}, \ref{nuvoloso}, \ref{zarro},  \ref{giove},  \ref{eulerissimo}, and \ref{luxlux} and their relative corollaries hold for $(M,g)$.
\end{cor}
Furthermore we mention that recently,  in his PhD thesis \cite{FLA},  Frank Lapp generalised  the result of Lesch  and Peyerimhoff to the following case: consider again a compact and oriented manifold with boundary $\overline{M}$ such that the boundary, that we still label with $N$, is diffeomorphic to a product of closed manifolds: $N\cong N_1\times...\times N_q$. Let $U$ be a collar neighborhood of $N$ and let $g$ be a riemannian metric on $M$ such that over $U\cong  [0,1)\times N_1\times...\times N_q$ it takes the form 
\begin{equation}
\label{cornicorni}
dx^2+h_1^2(x)g_{1}+...+h_q^2(x)g_q
\end{equation}
where $h_i(x)\in C^{\infty}((0,1], (0,\infty))$ and $h_1,...,h_q$ are riemannian metrics on $N_1,...,N_q$ respectively. A metric with this shape is called a \textbf{multiply warped product metric.} In his thesis, see  \cite{FLA} pag. 115 Theorem 5.3.5, Lapp proved  that if for some constant $K>0$ and $\beta>1$ 
\begin{equation}
\label{raffipa}
\max_{j=1,...,q}h_j(x)\leq Kr^{\beta}\ x\in (0,1)
\end{equation} and for every $j=1,...,q$ there exists a real number $c_j$ such that  
\begin{equation}
\label{martiemati}
\int_0^1x|\log x||\frac{h'_j(x)}{h_j(x)}-\frac{c_j}{x}|^2dx<\infty
\end{equation}
 then every closed extension of the Gauss-Bonnet operator $$d+\delta:L^2\Omega^*(M,g)\rightarrow L^2\Omega^*(M,g)$$ with domain given by $\Omega^*_c(M)$, is a Fredholm operator on its  domain endowed with the graph norm. In particular this is true when $g$ is a \textbf{multiply metric horns} that is in \eqref{cornicorni} all the warping functions satisfy the following requirement: $$h_i(x)=x^{\beta_i},\ \beta_i>1,\ i=1,...,q.$$ Therefore, using again Prop. \ref{worref} and Lemma 2.3 of \cite{BL}, we can conclude that {\em every closed extension of} $(\Omega_c^i(M),g)$ {\em is a Fredholm complex}. In particular, again according to Prop. \ref{topoamotore}, $\im(d_{min,i})$ is closed for each $i$. So we have the following corollary:

\begin{cor}
\label{BnB}
Let $(M,g)$ be an  oriented riemannian manifold where $M$ is the interior of a compact manifold with boundary $\overline{M}$. Suppose that the boundary is diffeomorphic to a product $\partial{\overline{M}}\cong N_1\times...\times N_q$. Let $g$ be  a riemannian metric on $M$ quasi-isometric to a   multiply warped product metric which satisfies condition \eqref{raffipa} and \eqref{martiemati}. Then Theorems \ref{polcas}, \ref{nuvoloso}, \ref{zarro},  \ref{giove},  \ref{eulerissimo}, and \ref{luxlux}  and their relative corollaries hold for $(M,g)$.
\end{cor}

Finally we conclude the paper with the following result. First of all we recall the definition of manifold with conical singularities:

\begin{defi}
Let $L$ be a  manifold. The truncated cone over $L$, usually labeled $C_{a}(L)$, is defined as 
\begin{equation}
\label{moria}
L\times [0,a)/(\{0\}\times L).
\end{equation}
\end{defi}

\begin{defi}
\label{gubbio}
A manifold with conical singularities $X$ is a metrizable, locally compact,  Hausdorff space such that there exists a sequence of points $\{p_{1},...,p_{n},...\}\subset X$ which satisfies the following properties:
\begin{enumerate}
\item $X\setminus\{p_{1},...,p_{n},...\}$ is a smooth  manifold.
\item For each $p_{i}$ there exists an open neighborhood $U_{p_i}$, a closed manifold $L_{p_i}$ and a  map $\chi_{p_i}:U_{p_i}\rightarrow C_{2}(L_{p_i})$ such that $\chi_{p_i}(p_i)=v$ and $ \chi_{p_i}|_{U_{p_i}\setminus\{p_{i}\}}:U_{p_i}\setminus\{p_i\}\rightarrow L_{p_i}\times (0,2)$  is a diffeomorphism. 
\end{enumerate}
\end{defi}

The regular and the singular part of $X$ are defined as $$\sing(X)=\{p_1,...,p_n,...\},\ \reg(X):=X\setminus \sing(X)=X\setminus \{p_1,...,p_n,...\}.$$The singular points $p_i$ are usually called \emph{conical points} and the smooth closed manifold $L_{p_i}$ is usually called the \emph{link} relative to the point $p_i$. If $X$  is compact then it is clear, from the above definition,  that the sequences of conical points $ \{p_1,...,p_n,...\}$ is made of isolated points  and therefore on $X$ there are just a finite number of conical points.\\Now we recall from \cite{ALMPI} a particular case, which is suitable for our purpose, of an important result which  describe a blowup process  to resolve the singularities.

\begin{prop} 
\label{palcano}
Let $X$ be a compact manifold with conical singularities. Then there exists a manifold with boundary $\overline{M}$ and a blow-down map $\beta:\overline{M}\rightarrow X$ which has the following properties: 
\begin{enumerate}
\item $\beta|_{M}:M\rightarrow \reg(X)$, where $M$ is the interior of $\overline{M}$, is a diffeomorphsim.
\item If $N$ is a connected component of $\partial \overline{M}$ and if  $U\cong  N\times [0,1)$ is a collar  neighborhood of $N$ then $\beta(U)=N\times [0,1)/(N\times \{0\})$. In particular $\beta(N)=p$ where $p$ is a conical point of $X$ and $N$ becomes one of the connected components of  the link of $p$.
\item If for each conical point $p_i$  the relative link $L_{p_i}$ is connected,  then there is a bijection between the conical points of $X$ and the connected components of $\partial \overline{M}.$
\end{enumerate}
\end{prop}

\begin{proof}
See \cite{ALMPI}, Proposition 2.5.
\end{proof}

Now we introduce a class of  riemannian  metrics  on these spaces.

\begin{defi}
\label{caiserra}
Let $X$ be a manifold with conical singularities. A conic metric $g$ on $\reg(X)$ is riemannian metric with the following property: for each  conical point $p_i$ there exists a map $\chi_{p_i}$, as defined in Definition \ref{gubbio},  such that \begin{equation}
\label{pianello}
(\chi_{p_i}^{-1})^*(g|_{U_{p_{i}}})=dx^2+x^2h_{L{p_{i}}}(x)
\end{equation}
 where $h_{L{p_{i}}}(x)$ depends smoothly on $x$ up to $0$ and for each fixed $x\in [0,1)$  it is a riemannian metric on $L_{p_{i}}.$ Analogously, if $\overline{M}$ is manifold with boundary and $M$ is its interior part, then $g$ is a conic metric on $M$ if it is a smooth, symmetric section of $T^{*}\overline{M}\otimes  T^*\overline{M}$, degenerate over the boundary, such that over a collar neighborhood $U$ of $\partial \overline{M}$, $g$ satisfies \eqref{pianello} with respect to some diffeomorphism $\chi:U\rightarrow [0,1)\times \partial \overline{M}.$
\end{defi}

Now consider again a compact and orientable manifold  $X$ with conical singularities such that $\reg(X)$ is endowed with a conic metric $g$.  Then from Definition \ref{gubbio}, Prop. \ref{palcano} and Def. \ref{caiserra} it is clear that Corollary \ref{coroll} applies to $(\reg(X), g)$. Moreover, as shown by Cheeger in \cite{JEC}, we have 
\begin{equation}
\label{orianaminor}
H^i_{2,max}(\reg(X),g)\cong I^{\underline{m}}H^i(X,\mathbb{R}),\ H^i_{2,min}(\reg(X),g)\cong I^{\overline{m}}H^i(X,\mathbb{R})
\end{equation} 

where $I^{\underline{m}}H^i(X,\mathbb{R})$ and  $I^{\overline{m}}H^i(X,\mathbb{R})$ are respectively the intersection cohomology groups of $X$ associated to the lower middle perversity and to the upper middle perversity. For the definition and the main properties of the intersection cohomology we refer to the fundamental papers \cite{GM} and \cite{GMA} or to the monographs \cite{BA} and \cite{KW}. Therefore we get the following corollary:
\begin{cor}
\label{marchionni}
Let $X$ be a compact and oriented manifold with conical singularities of  dimension $m$. Let $g$ be a conic metric on $\reg(X)$. Then: 
\begin{equation}
\label{birettina}
\psi_{L^2}(\reg(X),g)= \sum_{i=0}^m(-1)^i\dim(I^{\underline{m}}H^i(X,\mathbb{R}))-\sum_{i=0}^m(-1)^i\dim(I^{\overline{m}}H^i(X,\mathbb{R}))
\end{equation}
or equivalently
\begin{equation}
\label{bonabona}
\psi_{L^2}(\reg(X),g)=\left\{
\begin{array}{ll}
0\ &\ m\ is\ even\\
2\sum_{i=0}^m(-1)^i\dim(I^{\underline{m}}H^i(X,\mathbb{R}))\ &\ m\ is\ odd
\end{array}
\right.
\end{equation} 
Suppose now that $Y$ is another compact and oriented manifold with conical singularities. Let $h$ be a conic metric on $\reg(Y)$. Assume that $X$ and $Y$ are homeomorphic or that $X$ and $Y$ are equivalent through a  stratum preserving homotopy equivalences, (see \cite{KW} pag 62 for the definition of stratum preserving homotopy equivalences). Then: \begin{equation}
\label{mummis}
\psi_{L^2}(\reg(X),g)=\psi_{L^2}(\reg(Y),h).
\end{equation}
\end{cor}

\begin{proof}
As remarked above we can apply Cor.\ref{coroll} to $(\reg(X),g)$. Therefore $\psi_{L^2}(\reg(X),g)$ exists. Now combining Theorem \ref{zarro} with \eqref{orianaminor} we get \eqref{birettina} and  \eqref{bonabona}. Finally \eqref{mummis} follows by the invariance under homeomorphisms or under stratum preserving homotopy equivalences of the intersection cohomology groups.
\end{proof}

We conclude pointing out that, in the context of compact and oriented manifold with conical singularities, $\psi_{L^2}(\reg(X),g)$, defined using a conic metric $g$ on $\reg(X)$, admits a \textbf{pure topological interpretation} in terms of intersection cohomology groups of $X$.

\end{document}